\documentclass[11pt]{article}
\usepackage{amsmath,amsthm,amsfonts,amssymb,bbm}
\usepackage[]{natbib} 
\usepackage[pagebackref=true,linktoc=page,colorlinks=true,linkcolor=blue,citecolor=blue]{hyperref}
\usepackage[retainorgcmds]{IEEEtrantools} 

\usepackage{enumitem}
\usepackage{multirow} 
\usepackage{graphicx}
\usepackage[all]{xy}
\usepackage{tabularx} 
\usepackage{booktabs} 
\usepackage{accents} 
\usepackage[font=small,labelfont=bf]{caption} 
\usepackage{dsfont}
\usepackage{pstool,psfrag}

\usepackage{color} 

\newcommand{\RR}{\ensuremath{\mathbb{R}}}
\newcommand{\NN}{\ensuremath{\mathbb{N}}}

\newcommand{\EE}{\ensuremath{\mathbb{E}}}
\newcommand{\eps}{\ensuremath{\varepsilon}}
\newcommand{\borel}{\ensuremath{\mathcal{B}}}
\newcommand{\wk}{\ensuremath{\stackrel{\mathcal{D}}{\rightarrow}}}
\newcommand{\scale}{b}
\newcommand{\loc}{a}
\newcommand{\scalepm}{\scale_-}
\newcommand{\locpm}{\loc_-}
\newcommand{\scalemp}{\scale_+}
\newcommand{\locmp}{\loc_+}
\newcommand{\scaleOEtime}{\widehat{\scale}_t}
\newcommand{\locOEtime}{\widehat{\loc}_t}
\newcommand{\tpo}{\ensuremath{{t}}}
\newcommand{\fz}{\textbf{F}$_\mathbf{0}$}
\theoremstyle{plain}
\newtheorem{theorem}{Theorem}

\newtheorem{lemma}[theorem]{Lemma}

\theoremstyle{definition}

\newtheorem{example}{Example}

\theoremstyle{remark}
\newtheorem{remark}{Remark}

\begin{document}

\title{Extreme Events of Markov Chains}

\author{I.~Papastathopoulos\footnote{Postal address: University of Edinburgh, School of Mathematics, {Edinburgh EH9 3FD, UK}, Email address: i.papastathopoulos@ed.ac.uk},\, 
K.~Strokorb\footnote{Postal address: University of Mannheim,  Institute of Mathematics, 
 68131 Mannheim, Germany, Email address: strokorb@math.uni-mannheim.de},\, 
J.A.~Tawn\footnote{Postal address: Lancaster University, Department of Mathematics and Statistics, {Lancaster LA1 4YF}, UK, Email address: j.tawn@lancaster.ac.uk} \,
and A.~Butler\footnote{Postal address: Biomathematics and Statistics Scotland, \mbox{Edinburgh EH9 3FD}, UK, Email address: adam.butler@bioss.ac.uk}}

\enlargethispage{10mm}

\maketitle
\thispagestyle{empty}

\begin{abstract}
  The extremal behaviour of a Markov chain is typically characterized
  by its tail chain.\ For asymptotically dependent Markov chains
  existing formulations fail to capture the full evolution of the
  extreme event when the chain moves out of the extreme tail region
  and for asymptotically independent chains recent results fail to
  cover well-known asymptotically independent processes such as Markov
  processes with a Gaussian copula between consecutive values.  We use
  more sophisticated limiting mechanisms that cover a broader class of
  asymptotically independent processes than current methods, including
  an extension of the canonical Heffernan-Tawn normalization scheme,
  and reveal features which existing methods reduce to a degenerate
  form associated with non-extreme states.
\end{abstract}

{\small
  \phantom{\textit{2010 MSC}:} {Secondary 60G10} }


\section{Introduction}

Markov chains are natural models for a wide range of applications,
such as financial and environmental time series. For example, GARCH
models are used to model volatility and market crashes
\citep{mikoschstarica00,mikosch03,davismikosch09} and low order Markov
models are used to determine the distributional properties of cold
spells and heatwaves \citep{smithetal97,reichetal13,winttawn15} and
river levels \citep{easttawn12}.  It is the extreme events of the
Markov chain that are of most practical concern, e.g., for risk
assessment.  \cite{root88} showed that the extreme events of
stationary Markov chains that exceed a high threshold converge to a
Poisson process and that limiting characteristics of the values within
an extreme event can be derived, under certain circumstances, as the
threshold converges to the upper endpoint of the marginal
distribution. It is critical to understand better the behaviour of a
Markov chain within an extreme event under less restrictive conditions
through using more sophisticated limiting mechanisms. This is the
focus of this paper.

As pointed out by \cite{coleheff99} and \cite{ledtawn03}, when
analysing the extremal behaviour of a stationary process $\{X_t \,:\,
t=0,1,2,\dots \}$ with marginal distribution $F$, one has to
distinguish between two classes of extremal dependence that can be
characterized through the quantity
\begin{align}\label{eq:admeasure}
  \chi_t=\lim_{u \to 1}\Pr(F(X_t)>u \,|\, F(X_0)>u).
\end{align}
When $\chi_t>0$ for some $t>1$ ($\chi_t=0$ for all $t>1$) the process
is said to be \emph{asymptotically dependent} (\emph{asymptotically
  independent}) respectively. For a first order Markov chain, if
$\chi_1>0$, then $\chi_t>0$ for all $t>1$ \citep{Smith92}. For a broad
range of first order Markov chains we have considered, it follows that
when $\max(\chi_1,\chi_2)=0$, the process is asymptotically
independent at all lags. Here, the conditions on $\chi_1$ and $\chi_2$
limit extremal positive and negative dependence respectively.\ The
  most established measure of extremal dependence in stationary
  processes is the extremal index \citep{obri87}, denoted by $\theta$,
  which is important as $\theta^{-1}$ is the mean duration of the
  extreme event \citep{lead83}. In general $\chi_t$, for
  $t=1,2,\ldots $ does not determine $\theta$, however for first order
  Markov chains $\theta=1$ if $\max(\chi_1,\chi_2)=0$. In contrast
  when $\max(\chi_1,\chi_2)>0$ then we only know that $0<\theta<1$,
  with the value of $\theta$ determined by other features of the joint
  extreme behaviour of
  $(X_1,X_2,X_3)$.

To derive greater detail about within extreme events for Markov chains
we need to explore the properties of the \emph{tail chain} where a
tail chain describes the nature of the Markov chain after an extreme
observation, expressed in the limit as the observation tends to the
upper endpoint of the marginal distribution of $X_t$.  The study of
extremes of asymptotically dependent Markov chains by tail chains was
initiated by \cite{Smith92} and \cite{perf94} for deriving the value
of $\theta$ when $0<\theta<1$.  Extensions for asymptotically
dependent processes to higher dimensions can be found in \cite{perf97}
and \cite{janssege14} and to higher order Markov chains in
\cite{yun98} and multivariate Markov chains in
\citet{BasrakSegers09}. \cite{smithetal97}, \cite{sege07} and
\cite{janssege14} also study tail chains that go backwards in time and
\cite{perf94} and \cite{resnzebe13a} include regularity conditions
that prevent jumps from a non-extreme state back to an extreme state,
and characterisations of the tail chain when the process can suddenly
move to a non-extreme state.  Almost all the above mentioned tail
chains have been derived under regular variation assumptions on the
marginal distribution, rescaling the Markov chain by the extreme
observation resulting in the tail chain being a multiplicative random
walk.  Examples of statistical inference exploiting these results for
asymptotically dependent Markov chains are \cite{smithetal97} and
\citet{DreesSegersWarchol15}.

Tail chains of Markov chains whose dependence structure may exhibit
asymptotic independence were first addressed by Butler in the
discussion of \cite{hefftawn04} and \cite{Butler05}. More recently,
\cite{kulisoul} treat asymptotically independent Markov chains for
regularly varying marginal distributions of whose limiting tail chains
behaviour can be studied by a scale normalization using a regularly
varying function of the extreme observation and under assumptions that
prevent both jumps from a extreme state to a non-extreme state and
vice versa.

The aim of this article is to further weaken these limitations with an
emphasis on the asymptotic independent case. For example, the existing
literature fails to cover important cases such as Markov chains whose
transition kernel normalizes under the canonical family from
\cite{hefftawn04} nor applies to Gaussian copulas. Our new results
cover existing results and these important families as well as
inverted max-stable copulas \citep{ledtawn97}.  Furthermore, we are
able to derive additional structure for the tail chain, termed the
\emph{hidden tail chain}, when classical results give that the tail
chain suddenly leaves extreme states and also when the tail chain is
able to return to extremes states from non-extreme states. One key
difference in our approach is that, while previous accounts focus on
regularly varying marginal distributions, we assume our marginal
distributions to be in the Gumbel domain of attraction, like
\cite{Smith92}, as with affine norming this marginal choice helps to
reveal structure not apparent through affine norming of regularly
varying marginals.

To make this specific consider the distributions of $X^R_{t+1}|X^R_t$
and $X^G_{t+1}|X^G_t$, where $X_t^R$ has regularly varying tail and
$X_t^G$ is in the domain of attraction of the Gumbel distribution,
respectively, and hence crudely $X_t^G=\log(X_t^R)$. \cite{kulisoul}
consider non-degenerate distributions of
\begin{align}
  \lim_{x\rightarrow \infty} \Pr \bigg( \frac{X^R_{t+1}}{a^R(x)}
  <z \,\bigg\vert\, X^R_t>x \bigg)
\label{eq:kulsoul_cond}
\end{align}
with $a^R>0$ a regularly varying function. In contrast we consider the
non-degenerate limiting distributions of
\begin{align}
  \lim_{x\rightarrow \infty}
  \Pr\bigg(\frac{X^G_{t+1}-a(X^G_t)}{b(X^G_t)} <z \,\bigg\vert\,
  X^G_t > x \bigg)
\label{eq:ht_cond}
\end{align}
with affine norming functions $a$ and $b>0$. There are two differences between these limits: the use of random norming, using the previous value $X_t^G$  instead of a deterministic norming that uses the threshold $x$, and the use of affine norming functions $a$ and $b>0$ after a log-transformation instead of simply a scale norming $a^R$.  Under the framework of extended regular variation \cite{resnzebe14} give mild conditions which leads to limit~\eqref{eq:kulsoul_cond} existing with identical norming functions when  either random or deterministic norming is used. Under such conditions, when
limit~\eqref{eq:kulsoul_cond} is non-degenerate then
limit~\eqref{eq:ht_cond} is also non-degenerate with $a(\cdot)=\log a^R(\exp(\cdot))$ and $b(\cdot)=1$, whereas the converse does not hold when $b(x)\nsim 1$ as $x\rightarrow \infty$. In this paper we will illustrate a number of examples of practical importance where $b(x)\nsim 1$ as $x\rightarrow \infty$ for which the approach of \cite{kulisoul} fails but limit~\eqref{eq:ht_cond} reveals interesting structure.

\textbf{Organization of the paper.}  In Section~\ref{sec:tailchain},
we state our main theoretical results deriving tail chains with affine
update functions under rather broad assumptions on the extremal
behaviour of both asymptotically dependent and asymptotically
independent Markov chains. As in previous accounts
(\cite{perf94,resnzebe13a,janssege14} and \cite{kulisoul}), our
results only need the homogeneity (and not the stationarity) of the
Markov chain and therefore, we state our results in terms of
homogeneous Markov chains with initial distribution $F_0$ (instead of
stationary Markov chains with marginal distribution $F$).  We apply
our results to stationary Markov chains with marginal distribution
$F=F_0$ in Section~\ref{sec:examples} to illustrate tail chains for a
range of examples that satisfy the conditions of
Section~\ref{sec:tailchain} but are not covered by existing results.
In Section~\ref{sec:extensions} we derive the hidden tail chain for a
range of examples that fail to satisfy the conditions of
Section~\ref{sec:tailchain}. Collectively these reveal the likely
structure of Markov chains that depart from the conditions of
Section~\ref{sec:tailchain}.  All proofs are postponed to
Section~\ref{sec:proofs}.

\textbf{Some notation.} Throughout this text, we use the following
standard notation. For a topological space $E$ we denote its
\emph{Borel-$\sigma$-algebra} by $\borel(E)$ and the \emph{set of
  bounded continuous functions} on $E$ by $C_b(E)$. If $f_n,f$ are
real-valued functions on $E$, we say that $f_n$ (resp.\ $f_n(x)$)
\emph{converges uniformly on compact sets} (in the variable $x \in E$)
to $f$ if for any compact $C \subset E$ the convergence $\lim_{n \to
  \infty} \sup_{x \in C} \lvert f_n(x)-f(x) \rvert = 0$ holds
true. Moreover, $f_n$ (resp.\ $f_n(x)$) will be said to \emph{converge
  uniformly on compact sets to $\infty$} (in the variable $x \in E$)
if $\inf_{x \in C} f_n(x) \to \infty$ for compact sets $C \subset E$.
\emph{Weak convergence of measures} on $E$ will be abbreviated by
$\wk$. When $K$ is a distribution on $\RR$, we simply write $K(x)$
instead of $K((-\infty,x])$. If $F$ is a distribution function, we
abbreviate its survival function by $\overline{F}=1-F$ and its
generalized inverse by $F^\leftarrow$. The relation $\sim$ stands for
``is distributed like'' and the relation $\doteq$ means ``is
asymptotically equivalent to''.

\section{Statement of theoretical results}\label{sec:tailchain}

Let $\{X_t \,:\, t=0,1,2,\dots \}$ be a homogeneous real-valued Markov
chain with initial distribution $F_0(x)=\Pr(X_0 \leq x)$, $x \in \RR$
and transition kernel
\begin{align*}
  \pi(x,A)=\Pr(X_{t+1} \in A \mid X_t=x), \qquad x \in \RR,\, A \in
  \borel(\RR), \qquad t=0,1,2,\dots.
\end{align*}
There are many situations, where there exist suitable location and
scale norming functions $a(v)\in\RR$ and $b(v)>0$, such that the
normalized kernel $\pi(v, a(v)+b(v)dx)$ converges weakly to some
non-degenerate probability distribution as $v$ becomes large, cf.\
\cite{hefftawn04,resnzebe14} and Sections~\ref{sec:examples} and
\ref{sec:extensions} for several important examples. Note that the
normalized transition kernel $\pi(v, a(v)+b(v)dx)$ corresponds to the
random variable $(X_{t+1}-a(v))/b(v)$ conditioned on $X_t=v$. To
simplify the notation, we sometimes write
\begin{align*}
  \pi(x,y)=\Pr(X_{t+1} \leq y \mid X_t=x), \qquad x,y \in \RR, \qquad
  t=0,1,2,\dots.
\end{align*}
Our goal in this section is to formulate general (and practically
checkable) conditions that extend the convergence above (which
concerns only one step of the Markov chain) to the convergence of the
finite-dimensional distributions of the whole normalized Markov chain
\begin{align*}
  \left\{ \frac{X_t-\loc_t(X_0)}{\scale_t(X_0)} \,:\, t=1,2,\dots
  \right\} \,\bigg\vert\, X_0 > u
\end{align*}
to a tail chain $\{M_t \,:\, t=1,2,\dots \}$ as the threshold $u$
tends to its upper endpoint. Using the actual value $X_0$ as the
argument in the normalizing functions (instead of the threshold $u$),
is usually referred to as \emph{random norming} \citep{heffres07} and
is motivated by the belief that the actual value $X_0$ contains more
information than the exceeded threshold $u$.  It is furthermore
convenient that not only the normalization of the original chain
$\{X_t \,:\, t=1,2,\dots \}$ can be handled via \emph{location-scale}
normings, but if also the update functions of the tail chain $\{M_t
\,:\, t=1,2,\dots \}$ are \emph{location-scale} update functions. That
is, they are of the form $M_{t+1} =
\psi^\loc_\tpo(M_t)+\psi^\scale_\tpo(M_t) \,\eps_\tpo$ for an i.i.d.\
sequence of innovations $\{ \eps_\tpo \,:\, t=1,2,\dots \}$ and update
functions $\psi^\loc_\tpo(x)\in\RR$ and $\psi^\scale_\tpo(x)>0$.

The following assumptions on the extremal behaviour of the original
Markov chain $\{X_t \,:\, t=0,1,2,\dots \}$ make the above ideas
rigorous and indeed lead to location-scale tail chains in
Theorems~\ref{thm:tailchain} and~\ref{thm:tailchain:nonneg}. Our first
assumption concerns the extremal behaviour of the initial distribution
and is the same throughout this text.

\begin{description}[wide=0\parindent]
\item[Assumption \fz]
  \emph{(extremal behaviour of the initial distribution)}\\
  $F_0$ has upper endpoint $\infty$ and there exist a probability
  distribution $H_0$ on $[0,\infty)$ and a measurable norming function
  $\sigma(u)>0$, such that
  \begin{align*}
    \frac{F_{0}(u+\sigma(u) dx)}{\overline{F}_0(u)} \wk H_0(dx) \qquad
    \text{as } u \uparrow \infty.
  \end{align*}
\end{description}

We will usually think of ${H_0}(x)=1-\exp(-x)$, $x \geq 0$ being the
standard exponential distribution, such that $F_0$ lies in the Gumbel
domain of attraction. Next, we assume that the transition kernel
converges weakly to a non-degenerate limiting distribution under
appropriate location and scale normings. We distinguish between two
subcases.


\paragraph{First case (A) -- Real-valued chains with location
  and scale norming}
\begin{description}[wide=0\parindent]
\item[Assumption A1]
  \emph{(behaviour of the next state as the previous state becomes extreme)}\\
  There exist measurable norming functions $\loc(v)\in \RR$,
  $\scale(v)>0$ and a non-degenerate distribution function $K$ on
  $\RR$, such that
  \begin{align*}
    \pi(v, \loc(v) + \scale(v) dx) \wk K(dx) \qquad \text{as } v
    \uparrow \infty.
  \end{align*}
\end{description}

  \begin{remark}
    By saying that the distribution $K$ is supported on $\RR$, we do
    not allow $K$ to have mass at $-\infty$ or $+\infty$.  The weak
    convergence is meant to be on $\RR$.  In
    Section~\ref{sec:extensions} we will address situations in which
    this condition is relaxed.
  \end{remark}

  \begin{description}[wide=0\parindent]
  \item[Assumption A2] \emph{(norming functions and update functions
      for the tail chain)}
    \begin{enumerate}[label={(\alph*)}]
    \item Additionally to $\loc_1=\loc$ and $\scale_1=\scale$ there
      exist measurable norming functions $\loc_t(v) \in \RR$,
      $\scale_t(v)>0$ for each time step $t=2,3,\dots$, such that
      $\loc_t(v)+\scale_t(v)x \rightarrow \infty$ as $v \uparrow
      \infty$ for all $x \in \RR$, $t=1,2,\dots$.
    \item Secondly, there exist {continuous} update functions
      \begin{align*}
        \psi_\tpo^\loc(x) &= \lim_{v\rightarrow \infty}\frac{
          \loc\left(\loc_t(v)+\scale_t(v)x\right) -\loc_{t+1}(v)}{\scale_{t+1}(v)} \in \RR, \\
        \psi_\tpo^\scale(x) &= \lim_{v\rightarrow
          \infty}\frac{\scale\left(\loc_t(v)+\scale_t(v)x\right)}{\scale_{t+1}(v)}>0,
      \end{align*}
      defined for $x \in \RR$ and $t=1,2,\dots$, such that the
      remainder terms
      \begin{align*}
        r^\loc_\tpo(v,x)&=\frac{\loc_{t+1}(v) -
          \loc(\loc_t(v)+\scale_t(v)x) +\scale_{t+1}(v)
          \psi^\loc_\tpo(x)}{\scale(\loc_t(v)+\scale_t(v)x)},\\
        r^\scale_\tpo(v,x)&=1-\frac{\scale_{t+1}(v)
          \psi^\scale_\tpo(x)}{\scale(\loc_t(v)+\scale_t(v)x)}
      \end{align*}
      converge to $0$ as $v\uparrow \infty$ and both convergences hold
      uniformly on compact sets in the variable $x \in \RR$.
    \end{enumerate}
  \end{description}

\begin{remark}
  The update functions $\psi^\loc_\tpo$, $\psi^\scale_\tpo$ are
  necessarily given as in assumption \textbf{A2} if the remainder
  terms $r^\loc_\tpo$, $r^\scale_\tpo$ therein converge to $0$.
\end{remark}

\begin{theorem}
  \label{thm:tailchain}
  Let $\{X_t \,:\, t = 0,1,2, \dots \}$ be a homogeneous Markov chain
  satisfying assumptions \fz, \textbf{A1} and \textbf{A2}. Then, as $u
  \uparrow \infty$,
  \begin{align*}
    \left(\frac{X_0 -
        u}{\sigma(u)},\frac{X_1-\loc_1(X_0)}{\scale_1(X_0)},\frac{X_2-\loc_2(X_0)}{\scale_2(X_0)},\dots,\frac{X_t-\loc_t(X_0)}{\scale_t(X_0)}\right)
    \,\bigg\vert\, X_0 > u
  \end{align*}
  converges weakly to $\left(E_0,M_1,M_2,\dots,M_t\right)$, where
  \begin{enumerate}[label={(\roman*)}]
  \item $E_0 \sim H_0$ and $(M_1,M_2,\dots,M_t)$ are independent,
  \item $M_1 \sim K$ and $M_{t+1} =
    \psi^\loc_\tpo(M_t)+\psi^\scale_\tpo(M_t) \, \eps_\tpo, \,
    t=1,2,\dots$ for an i.i.d.\
    sequence 
    of innovations $\eps_\tpo \sim K$.
  \end{enumerate}
\end{theorem}

\begin{remark}\label{rk:tailchain:relaxed}
  Let $S_t=\{ x \in \RR \,:\, \Pr(M_t \leq x) > 0\}$ be the support of
  $M_t$
  and $\overline{S}_t$ its closure in $\RR$.  The conditions in
  assumption \textbf{A2} may be relaxed by replacing all requirements
  for ``$x \in \RR$'' by requirements for ``$x \in
  \overline{S}_t$'' if we assume the kernel convergence in assumption
  \textbf{A1} to hold true on $\overline{S}_1$, cf.\ also
  Remark~\ref{rk:tailchain:relaxed:proof} for modifications in the
  proof.
\end{remark}

\paragraph{Second case (B) -- Non-negative chains with only scale norming}
$\phantom{a}$\\
Considering non-negative Markov chains, where no norming of the
location is needed, requires some extra care, as the convergences in
assumption \textbf{A2} will not be satisfied anymore for all $x \in
[0,\infty)$, but only for $x \in (0,\infty)$. Therefore, we have to
control the mass of the limiting distributions at $0$ in this case.

\begin{description}[wide=0\parindent]
\item[Assumption B1]
  \emph{(behaviour of the next state as the previous state becomes extreme)}\\
  There exists a measurable norming function $\scale(v)>0$ and a
  non-degenerate distribution function $K$ on $[0,\infty)$ with no
  mass at $0$, i.e.\ $K(\{0\})=0$, such that
  \begin{align*}
    \pi(v, \scale(v) dx) \wk K(dx) \qquad \text{as } v \uparrow
    \infty.
  \end{align*}

\item[Assumption B2] \emph{(norming functions and update functions for
    the tail chain)}
  \begin{enumerate}[label={(\alph*)}]
  \item Additionally to $\scale_1=\scale$ there exist measurable
    norming functions $\scale_t(v)>0$ for $t=2,3,\dots$, such that
    $\scale_t(v) \rightarrow \infty$ as $v \uparrow \infty$ for all
    $t=1,2,\dots$.
  \item Secondly, there exist {continuous} update
    functions 
    \begin{align*}
      \psi_\tpo^\scale(x) &= \lim_{v\rightarrow
        \infty}\frac{\scale\left(\scale_t(v)x\right)}{\scale_{t+1}(v)}>0,
    \end{align*}
    defined for $x \in (0,\infty)$ and $t=1,2,\dots$, such that the
    following remainder term
    \begin{align*}
      r^\scale_\tpo(v,x)&=1-\frac{\scale_{t+1}(v)
        \psi^\scale_\tpo(x)}{\scale(\scale_t(v)x)}
    \end{align*}
    converges to $0$ as $v \uparrow \infty$ and the convergence holds
    uniformly on compact sets in the variable $x \in [\delta,\infty)$
    for any $\delta>0$.
  \item Finally, we assume that $\sup\{ x>0 \,:\,
    \psi^{\scale}_\tpo(x) \leq c \} \to 0$ as $c \downarrow 0$ with
    the convention that $\sup(\emptyset)=0$.
  \end{enumerate}

\end{description}

\begin{theorem}
  \label{thm:tailchain:nonneg}
  Let $\{X_t \,:\, t = 0,1,2, \dots \}$ be a non-negative homogeneous
  Markov chain satisfying assumptions \fz, \textbf{B1} and
  \textbf{B2}. Then, as $u \uparrow \infty$,
  \begin{align*}
    \left(\frac{X_0 -
        u}{\sigma(u)},\frac{X_1}{\scale_1(X_0)},\frac{X_2}{\scale_2(X_0)},\dots,\frac{X_t}{\scale_t(X_0)}\right)
    \,\bigg\vert\, X_0 > u
  \end{align*}
  converges weakly to $\left(E_0,M_1,M_2,\dots,M_t\right)$, where
  \begin{enumerate}[label={(\roman*)}]
  \item $E_0 \sim H_0$ and $(M_1,M_2,\dots,M_t)$ are independent,
  \item $M_1 \sim K$ and $M_{t+1} = \psi^\scale_\tpo(M_t) \,
    \eps_\tpo, \, t=1,2,\dots$ for an i.i.d.\
    sequence 
    of innovations $\eps_\tpo \sim K$.
  \end{enumerate}
\end{theorem}

\begin{remark}
  The techniques used in this setup can be used also for a
  generalisation of Theorem~\ref{thm:tailchain} in the sense that the
  conditions in assumption \textbf{A2} may be even further relaxed by
  replacing all requirements for ``$x \in \RR$'' by the respective
  requirements for ``$x \in S_t$'' (instead of ``$x \in
  \overline{S}_t$'' as in Remark~\ref{rk:tailchain:relaxed}) as long
  as it is possible to keep control over the mass of $M_t$ at the
  boundary of $S_t$ for all $t\geq 1$.  Some of the subtleties arising
  in such situations will be addressed by the examples in
  Section~\ref{sec:extensions}.
\end{remark}

\begin{remark}\label{rk:nonhom}
  The tail chains in Theorems~\ref{thm:tailchain} and
  \ref{thm:tailchain:nonneg} are potentially non-homogeneous since the
  update functions $\psi^a_\tpo$ and $\psi^b_\tpo$ are allowed to vary
  with $t$.
\end{remark}

\section{Examples}\label{sec:examples}

In this section, we collect examples of stationary Markov chains that
fall into the framework of Theorems~\ref{thm:tailchain} and
\ref{thm:tailchain:nonneg} with an emphasis on situations which go
beyond the current theory.  To this end, it is important to note that
the norming and update functions and limiting distributions in
Theorems~\ref{thm:tailchain} and \ref{thm:tailchain:nonneg} may vary
with the choice of the marginal scale. The following example
illustrates this phenomenon and is a consequence of
Theorem~\ref{thm:tailchain}.

\begin{example} \label{ex:marginalscale} \emph{(Gaussian transition kernel with Gaussian vs.\ exponential margins)}\\
  Let $\pi_G$ be the transition kernel arising from a bivariate
  Gaussian distribution with correlation parameter $\rho \in (0,1)$,
  that is
  \begin{align*}
    \pi_G(x,y) = \Phi\left(\frac{y-\rho x}{(1-\rho^2)^{1/2}}\right),
    \qquad \rho \in (0,1),
  \end{align*}
  where $\Phi$ denotes the distribution function of the standard
  normal distribution.  Consider a stationary Markov chain with
  transition kernel $\pi \equiv \pi_G$ and Gaussian marginal
  distribution $F=\Phi$. Then assumption \textbf{A1} is trivially
  satisfied with norming functions $\loc(v) = \rho v$ and
  $\scale(v)=1$ and limiting distribution
  $K_G(x)=\Phi((1-\rho^2)^{-1/2}x)$ on $\RR$.\ The normalization after
  $t$ steps
  $\loc_t(v) = \rho^t v,\,\scale_t(v)=1$
  yields the tail chain
  $M_{t+1} = \rho M_t + \varepsilon_\tpo$
  with $\eps_\tpo \sim K_G$.
  
  However, if this Markov chain is transformed to standard exponential
  margins, which amounts to changing the marginal distribution to
  $F(x)=1-\exp(-x)$, $x \in (0,\infty)$ and $(X_t,X_{t+1})$ having a
  Gaussian copula, then the transition kernel becomes
  \begin{align*}
    \pi(x,y) = \pi_G (
    \Phi^{\leftarrow}\{1-\exp(-x)\},\Phi^{\leftarrow}
    \{1-\exp(-y)\} ),
  \end{align*}
  and assumption \textbf{A1} is satisfied with different norming
  functions $\loc(v) = \rho^2 v$, $\scale(v)=v^{1/2}$ and limiting
  distribution $K(x)=\Phi(x/(2 \rho^2 (1-\rho^2))^{1/2})$ on $\RR$.
  \citep{hefftawn04}. A suitable normalization after $t$ steps is
  $\loc_t(v) = \rho^{2 t} v,\,\scale_t(v)= v^{1/2}$,
  which leads to the scaled autoregressive tail chain
  $M_{t+1} = \rho^2 M_t + \rho^t \varepsilon_\tpo$
  with $\eps_\tpo \sim K$.
\end{example}

To facilitate comparison between the tail chains obtained from
different processes, it is convenient therefore to work on a
prespecified marginal scale. This is in a similar vein to the study of
copulas \citep{Nelsen06,joe15}. Henceforth, we select this scale to be
standard exponential $F(x)=1-\exp(-x)$, $x \in (0,\infty)$, which
makes, in particular, the Heffernan-Tawn model class applicable to the
tail chain analysis of Markov chains as
follows. Theorems~\ref{thm:tailchain} and \ref{thm:tailchain:nonneg}
were motivated by this example.\ {It should be noted that the extremal
  index of any process is invariant to monotone increasing marginal
  transformations. Hence, our transformations enable assessment of the
  impact of different copula structure whilst not changing key
  extremal features.}

\begin{example} \label{ex:ht} \emph{(Heffernan-Tawn normalization)}\\
  \cite{hefftawn04} found that, working on the exponential scale, the
  weak convergence of the normalized kernel $\pi(v,a(v)+b(v)dx)$ to
  some non-degenerate probability distribution $K$
  is satisfied for transition kernels $\pi$ arising from various
  bivariate copula models if the normalization functions belong to the
  canonical family
  \begin{align*}
    \loc(v)=\alpha v,\, \scale(v)= v^\beta, \qquad (\alpha,\beta) \in
    [0,1] \times [0,1) \setminus \{(0,0)\}.
  \end{align*}
  The second Markov chain from Example~\ref{ex:marginalscale} with
  Gaussian transition kernel and exponential margins is an example of
  this type with $\alpha=\rho^2$ and $\beta=1/2$.  The general family
  covers different non-degenerate dependence situations and
  Theorems~\ref{thm:tailchain} and \ref{thm:tailchain:nonneg} allow us
  to derive the norming functions after $t$ steps and the respective
  tail chains as follows.

  \begin{enumerate}[label={(\roman*)}]
  \item If $\alpha=1$ and $\beta=0$, the normalization by $\loc_t(v)=
    v,\, \scale_t(v)=1$, yields the random walk tail chain
    $M_{t+1}=M_t+\varepsilon_\tpo$.
  
  \item If $\alpha \in (0,1)$ and $\beta \in [0,1)$, the normalization
    by $\loc_t(v)=\alpha^t v, \, \scale_t(v)=v^{\beta}$, gives the
    scaled autoregressive tail chain $M_{t+1} = \alpha M_t +
    \alpha^{t\beta} \varepsilon_\tpo$.
    
  \item If $\alpha=0$ and $\beta \in (0,1)$, the normalization by
    $\loc_t(v)=0,\,\scale_t(v)=v^{\beta^t}$, yields the exponential
    autoregressive tail chain $M_{t+1} = (M_t)^\beta
    \varepsilon_\tpo$.

  \end{enumerate}

  In all cases the i.i.d.\ innovations $\eps_\tpo$ stem from the
  respective limiting distribution $K$ of the normalized kernel $\pi$.
  Case (i) deals with Markov chains where the consecutive states are
  asymptotically dependent, cf.\ \eqref{eq:admeasure}. It is covered
  in the literature usually on the Fr{\'e}chet scale, cf.\
  \cite{perf94,resnzebe13a,kulisoul}. The other two cases are
  concerned with asymptotically independent consecutive states of the
  original Markov chain. Results of \cite{kulisoul} cover also the
  subcase of (ii), but only when $\beta=0$.
  In cases (i) and (ii), the location norming is dominant and
  Theorem~\ref{thm:tailchain} is applied, whereas, in case (iii), the
  scale norming takes over and Theorem~\ref{thm:tailchain:nonneg} is
  applied. Unless $\beta=0$, case (ii) yields a non-homogeneous tail
  chain and the remainder term related to the scale
  $r^\scale_\tpo(v,x)=O\left( v^{\beta-1}\right)$ in assumption
  \textbf{A2} does not vanish already for $v<\infty$.  It is worth
  noting that in all cases $\loc_{t+1}=\loc \circ \loc_t$ and in the
  third case (iii), when the location norming vanishes, also
  $\scale_{t+1}=\scale \circ \scale_t$.
\end{example}

Even though all transition kernels arising from the bivariate copulas
as given by \cite{heff00} and \cite{joe15} stabilize under the
Heffernan-Tawn normalization, it is possible that more subtle normings
are necessary.
\cite{paptawn15} found such situations for the bivariate inverted
max-stable distributions. The corresponding transition kernel
$\pi_{inv}$ on the exponential scale is given by
\begin{align*}
  \pi_{inv}(x,y) &= 1 + V_1(1,x/y) \exp\left(x - x V(1,x/y)\right),
\end{align*}
where the exponent measure $V$ admits
\begin{align*}
  V(x,y) = \int_{[0,1]} \max\{w/x,(1-w)/y\} H(dw)
\end{align*}
with $H$ being a Radon measure on $[0,1]$ with total mass 2 satisfying
the moment constraint $\int_{[0,1]}w~H(dw)=1$. The function $V$ is
assumed differentiable and $V_1(s,t)$ denotes the partial derivative
$\partial V(s,t)/\partial s$.  For our purposes, it will even suffice
to assume that the measure $H$ posseses a density $h$ on $[0,1]$. In
particular, it does not place mass at $\{0\}$, i.e., $H(\{0\})=0$.
Such inverted max-stable distributions form a class of models which
help to understand various norming situations.  In the following
examples, we consider stationary Markov chains with transition kernel
$\pi\equiv\pi_{inv}$ and exponential margins.  First, we describe two
situations, in which the Heffernan-Tawn normalization applies.

\begin{example} \emph{(Examples of the Heffernan-Tawn normalization
    based on inverted max-stable distributions)}
  \begin{enumerate}[label={(\roman*)}]
  \item 
    If the density $h$ satisfies $h(w) \doteq \kappa w^s$ as
    $w\downarrow 0$ for some $s>-1$, the Markov chain with transition
    kernel $\pi_{inv}$ can be normalized by the Heffernan-Tawn family
    with $\alpha=0$ and $\beta = (s+1)/(s+2) \in (0,1)$
    \citep{hefftawn04}.
  \item 
    If $\ell \in (0,1/2)$ is the lower endpoint of the measure $H$ and
    its density $h$ satisfies $h(w) \doteq \kappa (w-\ell)^s$ as
    $w\downarrow \ell$ for some $s>-1$, the Markov chain with
    transition kernel $\pi_{inv}$ can be normalized by the
    Heffernan-Tawn family with $\alpha = \ell/(1-\ell) \in (0,1)$ and
    $\beta = (s+1)/(s+2) \in (0,1)$ \citep{paptawn15}.
  \end{enumerate}
  In both cases the temporal location-scale normings and tail chains
  are as in Example~\ref{ex:ht}.
\end{example}

The next examples require more subtle normings than the Heffernan-Tawn
family. We also provide their normalizations after $t$ steps and the
respective tail chains. The relations $\loc_{t+1}(v) \doteq \loc \circ
\loc_t(v)$ and $\scale_{t+1}(v) \doteq \scale \circ \scale_t(v)$ hold
asymptotically as $v \uparrow \infty$ in these cases. {In each case
  for all $t$, $a_t(x)$ is regularly varying with index 1, i.e.,
  $a_t(x) =x{\cal L}_t(x)$, where ${\cal L}_t$ is a slowly varying
  function and the process is asymptotically independent. This seems
  contrary to the canonical class of Example 2 (i) where when
  $a_t(x)=x$ the process was asymptotically dependent. The key
  difference however is that as $x\uparrow \infty$, ${\cal L}_t(x)
  \downarrow 0$, so $a_t(x)/x\downarrow 0$ as $x\uparrow \infty$
  for all $t$ and hence subsequent values of the process are
  necessarily of smaller order than the first large value in the chain.}

\begin{example} \emph{(Examples beyond the Heffernan-Tawn
    normalization based on inverted max-stable distributions)}
  \begin{enumerate}[label={(\roman*)}]
  \item \emph{(Inverted max-stable copula with H\"usler-Reiss resp.\ Smith dependence)}\\
    If the exponent measure $V$ is the dependence model (cf.\
    \cite{huslreis89} Eq.~(2.7) or \cite{smit90b} Eq.~(3.1))
    \begin{align*}
      V(x,y) = \frac{1}{x}\Phi\left(\frac{\gamma}{2} +
        \frac{1}{\gamma}\log\left(\frac{y}{x}\right)\right) +
      \frac{1}{y}\Phi\left(\frac{\gamma}{2} +
        \frac{1}{\gamma}\log\left(\frac{x}{y}\right)\right)
    \end{align*}
    for some $\gamma>0$, then assumption \textbf{A1} is satisfied with
    the normalization
    \begin{align*}
      \loc(v)= v \,\exp\left(-\gamma(2\log v)^{1/2}+\gamma \frac{\log
          \log v}{\log v} + \gamma^2/2\right),\, \scale(v)=\loc(v)/
      (\log v)^{1/2}
    \end{align*}
    and limiting distribution $K(x) = 1 - \exp\left(- (8
      \pi)^{-1/2}\gamma \exp\left(\sqrt{2}x/\gamma\right)\right)$
    \citep{paptawn15}. The normalization after $t$ steps
    \begin{align*}
      & \loc_t(v)= v \,\exp\left(-\gamma t(2\log v)^{1/2} +\gamma t
        \frac{\log \log v}{\log v} + (\gamma t)^2/2\right),\,
      \scale_t(v)= \loc_t(v) / (\log v)^{1/2}
    \end{align*}
    yields, after considerable manipulation, the random walk tail
    chain
    \begin{align*}
      M_{t+1}=M_t+\varepsilon_\tpo
    \end{align*}
    with remainder terms $r^\loc_\tpo(v,x)=O\left((\log
      v)^{-1/2}\right), r^\scale_\tpo(v,x)=O\left((\log
      v)^{-1/2}\right)$.
  \item \emph{(Inverted max-stable copula with different type of decay)}\\
    If the density $h$ satisfies $h(w)\doteq w^\delta
    \exp\left(-\kappa w^{-\gamma}\right)$ as $w\downarrow 0$, where
    $\kappa,\gamma>0$ and $\delta \in \RR$, then assumption
    \textbf{A1} is satisfied with the normalization
    \begin{align*}
      \loc(v)= v\left(\frac{\log v}{\kappa}\right)^{-1/\gamma}\left(1
        + (c/\gamma^2)\,\frac{\log \log v}{ \log
          v}\right),\,\scale(v)=a(v)/\log v,
    \end{align*}
    where $c=\delta + 2(1+\gamma)$ and limiting distribution $K(x) = 1
    - \exp\left\{-c \exp(\gamma x)\right\}$ 
    \citep{paptawn15}. Set $\zeta_t = \binom{t}{t-2} +
    \binom{t}{t-1}\,c$, $t\geq 2$. Then the normalization after $t$
    steps
    \begin{align*}
      \loc_t(v)= v\left(\frac{\log
          v}{\kappa}\right)^{-t/\gamma}\left(1 +
        (\zeta_t/\gamma^2)\,\frac{\log \log v}{ \log v}\right),\,
      \scale_t(v)=a_t(v)/\log v
    \end{align*}
    yields, after considerable manipulation, the random walk tail
    chain with drift
    \begin{align*}
      M_{t+1}= M_t - (t/\gamma^2)\log \kappa + \varepsilon_\tpo
    \end{align*}
    with remainder terms $r^\loc_\tpo(v,x)=O\left((\log\log v)^2/(\log
      v)\right)$, $r^\scale_\tpo(v,x)=O\left(\log\log v/\log
      v\right)$.
  \end{enumerate}
\end{example}

{Note that in Example 4 each of the tail chains is a random walk (with
  possible drift term), like for the asymptotically dependent case of
  Example 2 (i). This feature is unlike Examples 2 (ii) and (iii)
  which though also asymptotically independent processes have
  autoregressive tail chains. This shows that Example 4 illustrates
  two cases in a subtle boundary class where the norming functions are
  consistent with the asymptotic independence class and the tail chain
  is consistent with the asymptotic dependent class.}

To give an impression of the different behaviours of Markov chains in
extreme states Figure~\ref{fig:tail_chains} presents properties of the
sample paths of chains for an asymptotically dependent and various
asymptotically independent chains. These Markov chains are stationary
with unit exponential marginal distribution and are initialised with
$X_0=10$, the $1-4.54\times 10^{-5}$ quantile. In each case the copula
of $(X_t,X_{t+1})$ for the Markov chain is in the Heffernan-Tawn model
class with transition kernels and associated parameters $(\alpha,
\beta)$ as follows:

\begin{enumerate}[label={(\roman*)}]
\item Bivariate extreme value (BEV) copula, with logistic dependence
  and transition kernel $ \pi(x,y) = \pi_F(T(x),T(y)), $ where
  $ T(x)=-1/\log\left(1 - \exp(-x)\right) $ and
  \begin{align*}
    \pi_F(x,y) =
    \bigg\{1+\Big(\frac{y}{x}\Big)^{-1/\gamma}\bigg\}^{\gamma-1}
    \exp\left\{-\left(x^{-1/\gamma} +
        y^{-1/\gamma}\right)^\gamma\right\}
  \end{align*}
  with $\gamma=0.152$. The chain is asymptotically dependent, i.e.,
  $(\alpha, \beta)=(1,0)$.
\item Inverted BEV copula with logistic dependence and transition
  kernel
  \begin{align*}
    \pi(x,y) = 1-
    \bigg\{1+\Big(\frac{y}{x}\Big)^{1/\gamma}\bigg\}^{\gamma-1}
    \exp\left\{x-\left(x^{1/\gamma} +
        y^{1/\gamma}\right)^\gamma\right\}
  \end{align*}
  with $\gamma=0.152$.  The chain is asymptotically independent with
  $(\alpha, \beta)=(0,1-\gamma)$.
\item Exponential auto-regressive process with constant slowly varying
  function \cite[p.~285]{kulisoul} and transition kernel
  \begin{align*}
    \pi(x,y) = \left(1-\exp\left[ - \left\{U(y) - \phi
          U(x)\right\}\right]\right)_+
  \end{align*}
  where $U(x)=F^{\leftarrow}_V(1-\exp(-x))$ and $F_V$ is a
  distribution function satisfying $ F_V(y) = 1 -
  \int_{-1/(1-\phi)}^{(y+1)/\phi} \exp\left\{-\left(y-\phi
      x\right)\right\}F_V(dx) $ for all $y>-1/(1-\phi)$ with
  $\phi=0.8$. The chain is asymptotically independent with $(\alpha,
  \beta)=(\phi,0)$.
\item Gaussian copula with correlation parameter $\rho=0.8$. The chain
  is asymptotically independent with $(\alpha, \beta)=(\rho^2,1/2)$.
\end{enumerate}

\begin{figure}[h!]
  \centering
  \includegraphics[scale=0.8]{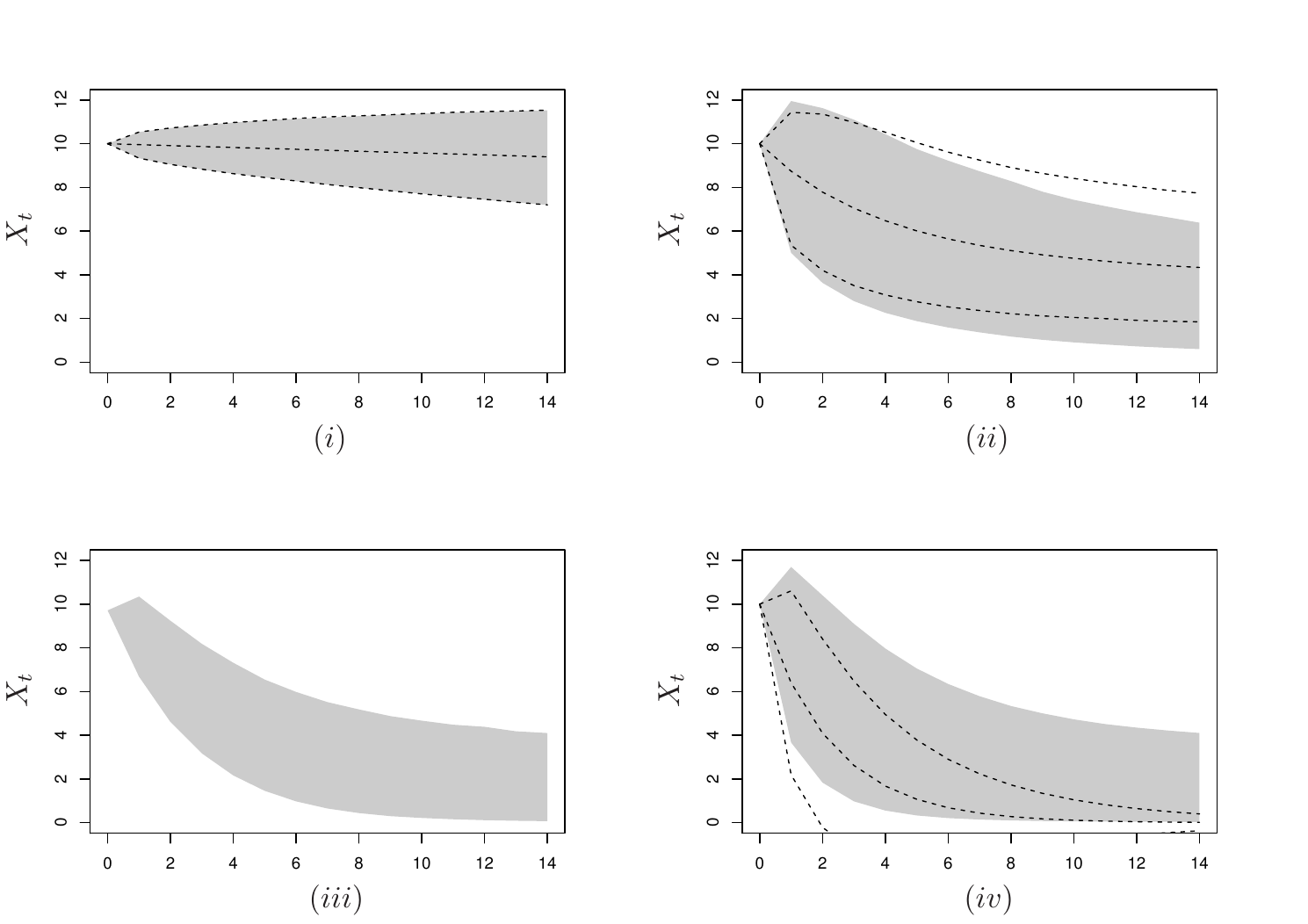}
  \caption{Four Markov chains in exponential margins 
    with different dependence structure and common initial extreme value of
    $x_0=10$. 
    Presented for each chain are: 
    2.5\% and 97.5\% quantiles of the actual chain $\{X_t\}$ started
    from $x_0=10$ (grey region); 2.5\% quantile, mean and 97.5\%
    quantile of the approximating chain $\{X^{TC}_t\}$ arising from
    the tail chain with $x_0=10$ (dashed lines, apart from (iii)).
    The copula of $(X_t,X_{t+1})$ comes from: (i): BEV copula, with
    logistic dependence structure, $\gamma=0.152$, (ii): inverted BEV
    copula with logistic dependence structure, $\gamma=0.152$, (iii):
    exponential auto-regressive process with $\phi=0.8$, (iv):
    Gaussian copula with $\rho=0.8$.}
  \label{fig:tail_chains}
\end{figure}

The parameters for chains (ii) and (iv) have been chosen such that the
coefficient of tail dependence \cite{ledtawn97} of the bivariate
margins is the same.  The plots compare the actual Markov chain
$\{X_t\}$ started from $X_0=10$ with the paths $\{X^{TC}_t\}$ arising
from the tail chain approximation $X_{t}^{TC} = a_t(X_0) +b_t(X_0)
M_t$, where $a_t$, $b_t$ and $M_t$ are as defined in
Example~\ref{ex:ht} and determined by the associated value of
$(\alpha,\beta)$ and the respective limiting kernel $K$.  The figure
shows both the effect of the different normalizations on the sample
paths and that the limiting tail chains provide a reasonable
approximation to the tail chain for this level of $X_0$, at least for
the first few steps.  Unfortunately, we were not able to derive the
limiting kernel $K$ from (iii) and so the limiting tail chain
approximation $\{X^{TC}_t\}$ is not shown in this case. Also note that
for the asymptotically independent processes and chain (iv) in
particular, there is some discrepancy between the actual and the
approximating limiting chains.  This difference is due to the slow
convergence to the limit here, a feature identified in the
multivariate context by \cite{hefftawn04} for chain (iv), but this
property can occur similarly for asymptotically dependent processes.

\section{Extensions}\label{sec:extensions}

In this section, we address several phenomena which have not yet been
covered by the preceding theory.  The information stored in the value
$X_0$ is often not good enough for assertions on the future due to
additional sources of randomness that influence the return to the body
of the marginal distribution or switching to a negative extreme state.
Let us assume, for instance, that the transition kernel of a Markov
chain encapsulates different modes of normalization.  If we use our
previous normalization scheme matching the dominating mode, the tail
chain will usually terminate in a degenerate state.  In order to gain
non-degenerate limits which allow for a refined analysis in such
situations, we will introduce random change-points that can detect the
misspecification of the norming and adapt the normings accordingly
after change-points. The first of the change-points plays a similar
role to the extremal boundary in \cite{resnzebe13a}. We also use this
concept to resolve some of the subtleties arising from random negative
dependence. The resulting limiting processes $\{M_t \,:\, t =
1,2,\dots \}$ of
\begin{align*}
  \left\{ \frac{X_t-\loc_t(X_0)}{\scale_t(X_0)} \,:\, t=1,2,\dots
  \right\} \,\bigg\vert\, X_0 > u
\end{align*}
as $u\uparrow \infty$ (with limits meant in finite-dimensional
distributions) will be termed \emph{hidden tail chains} if they are
based on change-points and adapted normings, even though $\{M_t\}$
need not be first order Markov chains anymore due to additional
sources of randomness in their update schemes. However, they reveal
additional (``hidden'') structure after certain change-points.  We
present such phenomena in the sequel by means of some examples which
successively reveal increasing complex structure. Weak convergence
will be meant on the extended real line including $\pm \infty$ if mass
escapes to these states.

\subsection{Hidden tail chains}

\paragraph{Mixtures of different modes of normalization}

\begin{example}\label{ex:asymm_BEV} \emph{(Bivariate extreme value copula with asymmetric logistic dependence)}\\
  The transition kernel $\pi_F$ arising from a bivariate extreme value
  distribution with asymmetric logistic distribution on Fr\'echet
  scale \citep{tawn88} is given by
  \begin{align*}
    \pi_{F}(x,y) = -x^2 \frac{\partial}{\partial x}V(x,y)
    \exp\bigg(\frac{1}{x}-V(x,y)\bigg),
  \end{align*}
  where $V(x,y)$ is the exponent function
  \begin{align*}
    V(x,y)= \frac{1-\varphi_1}{x} + \frac{1-\varphi_2}{y} +
    \bigg\{\Big(\frac{\varphi_{1}}{x}\Big)^{1/\nu} +
    \Big(\frac{\varphi_{2}}{y}\Big)^{1/\nu} \bigg\}^\nu, \qquad
    \varphi_1,\varphi_2,\nu \in (0,1).
  \end{align*}
  Changing the marginal scale from standard Fr{\'e}chet to standard
  exponential margins yields the transition kernel
  \begin{align*}
    \pi(x,y)=\pi_{F}(T(x),T(y)), \quad \text{where} \quad
    T(x)=-1/\log\left(1 - \exp(-x)\right).
  \end{align*}
  The kernel $\pi$ converges weakly with two distinct normalizations
  \begin{align*}
    \pi(v,v+dx) \wk K_1(dx) \quad \text{and} \quad \pi(v,dx) \wk
    K_2(dx) \qquad \text{as $v \uparrow \infty$}
  \end{align*}
  to the distributions
  \begin{align*}
    K_1&=(1-\varphi_{1}) \delta_{-\infty} + \varphi_1G_1, \qquad \, G_1(x)= \bigg[1 + \bigg\{\frac{\varphi_{2}}{\varphi_{1}} \exp(-x)\bigg\}^{1/\nu}\bigg]^{\nu-1}\\
    K_2&= (1-\varphi_{1})F_E+ \varphi_1 \delta_{+\infty}, \qquad
    F_{E}(x)=\left(1-\exp(-x)\right)_+
  \end{align*}
  with entire mass on $[-\infty,\infty)$ and $(0,\infty]$,
  respectively.  In the first normalization, mass of the size
  $1-\varphi_1$ escapes to $-\infty$, whereas in the second
  normalization the complementary mass $\varphi_1$ escapes to
  $+\infty$ instead. The reason for this phenomenon is that both
  normalizations are related to two different modes of the conditioned
  distribution of $X_{t+1}\mid X_{t}$ of the Markov chain, cf.\
  Figure~\ref{fig:as_log}. However, these two modes can be separated,
  for instance, by any line of the form $(x_t,c x_t)$ for some $c \in
  (0,1)$ as illustrated in Figure~\ref{fig:as_log} with $c=1/2$.  This
  makes it possible to account for the mis-specification in the two
  normings above by introducing the change-point
  \begin{align}
    T^X = \inf \left\{ t \geq 1 \,:\, X_{t} \leq c X_{t-1} \right\},
    \label{eq:changepoint}
  \end{align}
  i.e., $T^X$ is the first time that $c$ times the previous state is
  not exceeded anymore.  Adjusting the above normings to
  \begin{align*}
    a_t(v) &= \begin{cases}v & \quad t < T^X, \\ 0 & \quad t \geq
      T^X, \end{cases} \quad \text{and}\quad b_t(v)=1,
  \end{align*}
  yields the following hidden tail chain, which is built on an
  independent i.i.d.\ sequence $\{B_t \,:\, t = 1,2, \dots \}$ of
  latent Bernoulli random variables $B_t \sim \text{Ber}(\varphi_1)$
  and the hitting time $T^{B}= \inf \{ t \geq 1 \,:\, B_t = 0 \}$.
  Its initial distribution is given by
  \begin{align*}
    \Pr(M_{1} \leq x) =
    \begin{cases}
      G_1(x)
      & \quad T^{B}>1,\\
      F_E(x) & \quad T^{B}=1,
    \end{cases}
  \end{align*}
  and its transition mechanism is
  \begin{align*}
    \Pr(M_{t} \leq y \mid M_{t-1}=x) =
    \begin{cases}
      G_1(y-x)
      & \quad t< T^{B},\\
      F_E(y) & \quad t=T^{B},\\
      \pi(x,y) & \quad t > T^{B}.
    \end{cases}
  \end{align*}
  In other words, the tail chain behaves like a random walk with
  innovations from $K_1$ as long as it does not hit the value
  $-\infty$ and, if it does, the norming changes instead, such that
  the original transition mechanism of the Markov chain is started
  again from an independent exponential random variable.
\end{example}

  \begin{figure}[htpb!]
    \centering
    \includegraphics[scale=0.8]{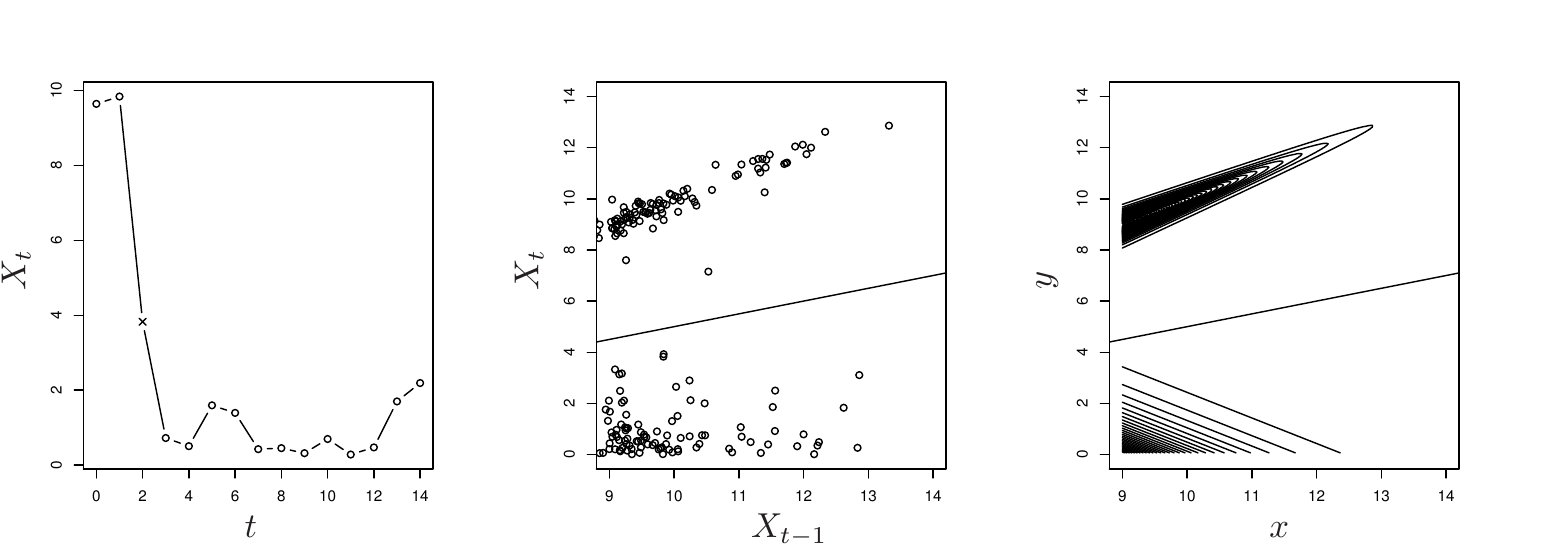}
    \caption{Left: time series plot showing a single realisation from
      the Markov chain with asymmetric logistic dependence,
      initialised from the distribution $X_0 \mid X_0>9$.  The
      change-point $T^X=2$ with $c=1/2$ (cf.\
      Eq.~\eqref{eq:changepoint}) is highlighted with a cross.\
      Centre: scatterplot of consecutive states $(X_{t-1},X_{t})$, for
      $t=1,\hdots,T^X$ with $c=1/2$, drawn from 1000 realisations of
      the Markov chain {initialised from $X_0 \mid X_0>9$} and line
      $X_{t}=X_{t-1}/2$ superposed.\ Right: Contours of joint density
      of asymmetric logistic distribution with exponential margins and
      line $y=x/2$ superposed.  The asymmetric logistic parameters
      used are $\varphi_1=\varphi_2=0.5$ and $\gamma=0.152$.}
    \label{fig:as_log}
  \end{figure}

In Example~\ref{ex:asymm_BEV} the adjusted tail chain starts as a
random walk and then permanently terminates in the transition
mechanism of the original Markov chain after a certain change-point
that can distinguish between two different modes of normalization.
These different modes arise as the conditional distribution 
of $X_{t+1}|X_t$ is essentially a mixture
distribution when $X_t$ is large with one component of the mixture
returning the process to a non-extreme state.

{ The following example extends this mixture structure
    to the case where both components of the mixture keep the process
    in an extreme state, but with different Heffernan and Tawn
    canonical family norming needed for each component. The first
    component gives the strongest form of extremal dependence. The
    additional complication that this creates is that there is now a
    sequence of change-points, as the process switches from one
    component to the other, and the behaviour of the resulting tail
    chain subtly changes between these.}

\begin{example}\label{ex:ht-mix}
  \emph{(Mixtures from the canonical Heffernan-Tawn model)}\\
  For two transition kernels $\pi_1$ and $\pi_2$ on the standard
  exponential scale, each stabilizing under the Heffernan-Tawn
  normalization
  \begin{align*}
    \pi_1(v,\alpha_1 v+v^{\beta_1}dx) \wk G_1(dx) \quad \text{and}
    \quad \pi_2(v,\alpha_2 v+v^{\beta_2}dx) \wk G_2(dx)
  \end{align*}
  as in Example~\ref{ex:ht} (ii) for $v \uparrow \infty$, let us
  consider the mixed transition kernel
  \begin{align*}
    \pi = \lambda \, \pi_1 + (1-\lambda) \, \pi_2, \qquad \lambda \in
    (0,1).
  \end{align*}
  Assuming that $\alpha_1>\alpha_2$, the kernel $\pi$ converges weakly
  on the extended real line with the two distinct normalizations
  \begin{align*}
    \pi(v,\alpha_1 v+v^{\beta_1}dx) \wk K_1(dx) \quad \text{and} \quad
    \pi(v,\alpha_2 v+v^{\beta_2}dx) \wk K_2(dx) \qquad \text{as $v
      \uparrow \infty$}
  \end{align*}
  to the distributions $K_1 = \lambda G_1 +
  (1-\lambda)\delta_{-\infty}$ and $K_2= (1-\lambda) G_2 + \lambda
  \delta_{+\infty}$,
  with mass $(1-\lambda)$ escaping to $-\infty$ in the first case and
  complementary mass $\lambda$ to $+\infty$ in the second
  case. Similarly to Example~\ref{ex:asymm_BEV}, the different modes
  of normalization for the consecutive states $(X_t,X_{t+1})$ are
  increasingly well separated by any line of the form $(x_t, {cx_t})$
  with $c \in (\alpha_2,\alpha_1)$. In this situation, the following
  recursively defined sequence of change-points
  \begin{align*}
    T_1^X &= \inf \left\{ t \geq 1 \,:\, X_{t} \leq c X_{t-1} \right\}\\
    T_{k+1}^X &= \begin{cases}
      \inf \left\{ t \geq T^X_k+1 \,:\, X_{t} > c X_{t-1} \right\} & \quad k \text{ odd},\\
      \inf \left\{ t \geq T^X_k+1 \,:\, X_{t} \leq c X_{t-1} \right\}
      & \quad k \text{ even}
    \end{cases}
  \end{align*}
  and the normings
  \begin{align*}
    \loc_t(v) = n^\alpha_t v, \quad \scale_t(v) =
    \begin{cases}
      v^{\beta_1} & \quad t <  T^X_1, \\
      v^{\beta_2} & \quad T^X_1=1 \text{ and } t <  T^X_2, \\
      v^{\max\{\beta_1,\beta_2\}} & \quad t \geq T^X_1, \text{ unless
      } T^X_1=1 \text{ and } t<T^X_2
    \end{cases}
  \end{align*}
  with
  \begin{align*}
    n^\alpha_t =
    \begin{cases}
      \alpha^t_1 & t < T^X_1,\\
      \alpha_1^{(S_k^\text{odd}-1)-S_k^\text{even}}\alpha_2^{t+S_k^\text{even}-(S_k^\text{odd}-1)}& T^X_k\leq t < T^X_{k+1}, \, k \text{ odd},\\
      \alpha_1^{t+S_k^\text{odd}-S_k^\text{even}}\alpha_2^{S_k^\text{even}-S_k^\text{odd}}&
      T^X_k\leq t < T^X_{k+1}, \, k \text{ even},
    \end{cases}
  \end{align*}
  and
  \begin{align*}
    S_k^\text{odd/even}=\sum_{\substack{j=1,\dots,k, \\ j \text{
          odd/even}}} T_j^X
  \end{align*}
  leads to a variety of transitions into less extreme states,
  depending on the ordering of $\beta_2$ and $\beta_1$. As in
  Example~\ref{ex:asymm_BEV}, the hidden tail chain can be based again
  on a set of latent Bernoulli variables $\{B_t \,:\, t = 1,2,\dots\}$
  with $B_t \sim \text{Ber}(\lambda)$. It has the initial distribution
  \begin{align*}
    M_{1} \sim
    \begin{cases}
      G_1 & \quad T^{B}_1>1, \\
      G_2 & \quad T^{B}_1=1,
    \end{cases}
  \end{align*}
  and is not a first order Markov chain anymore, as its transition
  scheme takes the position among the change-points
  \begin{align*}
    T^B_1&=\inf\{t \geq 1 \,:\, B_t \neq B_{t-1}\}\\
    T^B_{k+1}&=\inf\{t \geq T^{B}_{k}+1 \,:\, B_t \neq B_{t-1}\},
    \quad k=1,2,\dots,
  \end{align*}
  into account as follows
  \begin{align*}
    M_{t+1} =
    \begin{cases}
      \alpha_1 M_t + (n^\alpha_t)^{\beta_1} \eps_t^{(1)} & \quad
      t+1<T^B_1
      \text{ or } T^B_k \leq t+1 < T^B_{k+1}, \, k \text{ even}, \, \beta_1 \geq \beta_2,\\[-1mm]
      & \quad \text{unless } T^B_1=1, \, t+1=T^B_2 \text{ and } \beta_1 > \beta_2,
      \\[-1mm]
      \alpha_2 M_t + (n^\alpha_t)^{\beta_2} \eps_t^{(2)} 
      & \quad T^B_1=1 \text{ and } t+1 < T^B_2, \, \beta_1 > \beta_2,\\[-1mm]
      & \quad \text{or } T^B_k \leq t+1 < T^B_{k+1}, \, k \text{ odd}, \, \beta_1 \leq \beta_2,\\[-1mm]
      & \quad \text{unless } t+1=T_1 \text{ and } \beta_1 < \beta_2,\\[-1mm]
      (n^\alpha_t)^{\beta_1} \eps_t^{(1)} 
      & \quad T^B_1=1 \text{ and } t+1=T^B_2, \, \beta_1>\beta_2,\\[-1mm]
      (n^\alpha_t)^{\beta_2} \eps_t^{(2)} & \quad t+1=T^B_1, \, \beta_1<\beta_2,\\[-1mm]
      \alpha_1 M_t & \quad T^B_k \leq t+1 < T^B_{k+1}, \, k \text{
        even}, \, \beta_1 < \beta_2,\\[-1mm]
      \alpha_2 M_t  & \quad T^B_k \leq t+1 < T^B_{k+1}, \, k \text{ odd}, \, \beta_1 > \beta_2,\\[-1mm]
      & \quad \text{unless } T^B_1=1,\, k=1 \text{ and } \beta_1 > \beta_2.
    \end{cases}
  \end{align*}
  The independent innovations are drawn from either $\eps_t^{(1)} \sim
  G_1$ or $\eps_t^{(2)} \sim G_2$.\ The hidden tail chain can
  transition into a variety of forms depending on the characteristics
  of the transition kernels $\pi_1$ and $\pi_2$. According to the
  ordering of the scaling power parameters $\beta_1,\beta_2$, the tail
  chain at the transition points can degenerate to a scaled value of
  the previous state or independent of previous values.
\end{example}

\paragraph{Returning chains}
Finally, we consider Markov processes which can return to extreme
states.  Examples include tail switching processes, i.e., processes
that are allowed to jump between the upper and lower tail of the
marginal stationary distribution of the process.  To facilitate
comparison, we use the standard Laplace distribution
\begin{align}
  F_L(x)=\left\{
    \begin{array}{ll}
      \frac{1}{2} \exp(x) & \quad x<0,\\
      1-\frac{1}{2}\exp(-x) & \quad x\geq 0. 
    \end{array} \right. 
  \label{eq:lap_mar}
\end{align}
as a common marginal, so that both lower and upper tail is of the same
exponential type.

\begin{example} \label{ex:sm_roo}
  \emph{(Rootz\'{e}n/Smith tail switching process with Laplace margins)}\\
  As in \cite{Smith92} and adapted to our chosen marginal scale,
  consider the stationary Markov process that is initialised from the
  standard Laplace distribution and with transition mechanism built on
  independent i.i.d.\ sequences of standard Laplace variables $\{L_t
  \,:\, t = 0,1,2, \dots \}$ and Bernoulli variables $\{B_t \,:\, t =
  0,1,2, \dots \}$ with $B_t \sim \text{Ber}(0.5)$ as follows
  \begin{align*}
    X_{t+1} = -B_t X_t + (1-B_t) L_t=
    \begin{cases}
      -X_t & \quad B_t=1, \\
      L_t & \quad B_t=0.
    \end{cases}
  \end{align*}
  The following convergence situations arise as $X_0$ goes to its
  upper or lower tail
  \begin{align*}
    X_1+X_0 \mid X_0=x_0 &\wk \begin{cases}
      0.5\,(\delta_0 + \delta_{+\infty}) & \,\,\quad x_0 \uparrow +\infty, \\
      0.5\,(\delta_{-\infty} + \delta_{0}) & \,\,\quad x_0 \downarrow
      -\infty,
    \end{cases}
    \\
    X_1\mid X_0=x_0 &\wk \begin{cases}
      0.5\,(\delta_{-\infty} + F_L)& \quad x_0 \uparrow +\infty, \\
      0.5\,(F_L + \delta_{\infty}) & \quad x_0 \downarrow -\infty,
    \end{cases}
  \end{align*}
  where, in addition to their finite components $\delta_0$ and $F_L$,
  the limiting distributions collect complementary masses at $\pm
  \infty$.  Introducing the change-point
  \begin{align*}
    T^X = \inf\{t \geq 1 \,:\, X_{t} \neq X_{t-1}\}
  \end{align*}
  and adapted time-dependent normings
  \begin{align*}
    \loc_t(v) = \begin{cases}
      (-1)^t v & \quad t < T^X, \\
      0 & \quad t \geq T^X,
    \end{cases}
    \quad \text{and} \quad \scale_{t} (v) = 1,
  \end{align*}
  leads to the tail chain
  \begin{align*}
    M_{t} = \begin{cases} 0 &  \quad t < T^X,\\
      X'_{t-T^X} & \quad t\geq T^X,
    \end{cases}
  \end{align*}
  where $\{X'_{t} \,:\, t=0,1,2,\dots\}$ is a copy of the original
  Markov chain $\{X_{t} \,:\, t=0,1,2,\dots\}$.
\end{example}

Example~\ref{ex:sm_roo} illustrates that the Markov chain can return
to the extreme states visited before the termination time, it strictly
alternates between $X_0$ and $-X_0$.  Similarly with
Example~\ref{ex:asymm_BEV}, the hidden tail chain permanently
terminates in finite time and the process jumps to a non-extreme event
in the stationary distribution of the process.  The next example shows
a tail switching process with non-degenerate tail chain that does not
suddenly terminate.

\begin{example} \label{ex:arch} \emph{(ARCH with Laplace margins)}\\
  In its original scale the ARCH(1) process $\{Y_t \,:\,
  t=0,1,2,\dots\}$ follows the transition scheme $Y_t = \left(\theta_0
    + \theta_1\, Y_{t-1}^2\right)^{1/2}W_t$ for some $\theta_0>0$,
  $0<\theta_1<1$ and an i.i.d.\ sequence $\{W_t\,:\,t=0,1,2,\dots\}$
  of standard Gaussian variables.  It can be shown that,
  irrespectively of how the process is initialised, it converges to a
  stationary distribution $F_\infty$, whose lower and upper tail are
  asymptotically equivalent to a Pareto tail, i.e.,
  \begin{align*}
    1 - F_{\infty}(x) = F_{\infty}(-x)\doteq c x ^{-\kappa} \qquad
    \text{as } x\uparrow\infty,
  \end{align*}
  for some $c,\kappa>0$ \citep{dehaetal89}.  Initialising the process
  from $F_{{\infty}}$ yields a stationary Markov chain, whose
  transition kernel becomes
  \begin{align*}
    \pi(x,y)=\Phi\left( \frac{F^{\leftarrow}_\infty
        (F_L(y))}{(\theta_0 + \theta_1
        (F_\infty^{\leftarrow}(F_L(x)))^2 )^{1/2}} \right)
  \end{align*}
  if the chain is subsequently transformed to standard Laplace
  margins.  It converges with two distinct normalizations
  \begin{align*}
    \pi(v, v+dx) \wk \begin{cases}
      K_+(dx) & \quad v \uparrow +\infty, \\
      K_-(dx) & \quad v \downarrow -\infty,
    \end{cases}
    \\
    \pi(v,-v+dx) \wk \begin{cases}
      K_-(dx)& \quad v \uparrow +\infty, \\
      K_+(dx) & \quad v \downarrow -\infty
    \end{cases}
  \end{align*}
  to the distributions $K_+=0.5(\delta_{-\infty}+G_+)$ and
  $K_-=0.5(G_-+\delta_{+\infty})$ with
  \begin{align*}
    G_+(x)=2\Phi\left(\frac{\exp(x/\kappa)}{\sqrt{\theta_1}}\right)-1
    \quad \text{and} \quad
    G_-(x)=2\Phi\left(-\frac{\exp(-x/\kappa)}{\sqrt{\theta_1}}\right).
  \end{align*}
  Here, the recursively defined sequence of change-points
  \begin{align*}
    T^X_1&=\inf\{t \geq 1 \,:\, \text{sign}(X_t) \neq \text{sign}(X_{t-1})\}\\
    T^X_{k+1}&=\inf\{t \geq T^{X}_{k}+1 \,:\, \text{sign}(X_t) \neq
    \text{sign}(X_{t-1})\}, \quad k=1,2,\dots,
  \end{align*}
  which documents the sign change, and adapted normings
  \begin{align*}
    \loc_t(v) = \begin{cases} v &\quad
      t< T^X_1 \text{ or } T^X_{k} \leq  t < T^{X}_{k+1},\, k \text{ even},\\
      -v & \quad \phantom{t< T^X_1 \text{ or }} T^X_{k} \leq t <
      T^X_{k+1},\, k \text{ odd,}
    \end{cases}
    \qquad \scale_{t} (v) = 1,
  \end{align*}
  lead to a hidden tail chain (which is not a first order Markov chain
  anymore) as follows. It is distributed like a sequence $\{M_t \,:\,
  t=1,2,\dots\}$ built on the change-points
  \begin{align*}
    T^B_1&=\inf\{t \geq 1 \,:\, B_t \neq B_{t-1}\}\\
    T^B_{k+1}&=\inf\{t \geq T^{B}_{k}+1 \,:\, B_t \neq B_{t-1}\},
    \quad k=1,2,\dots,
  \end{align*}
  of an i.i.d.\ sequence of Bernoulli variables $\{B_t \,:\, t = 1,2,
  \dots \}$ via the initial distribution
  \begin{align*}
    M_1 \sim \begin{cases}
      G_+ &\quad T^B_1>1,\\
      G_- &\quad T^B_1=1,
    \end{cases}
  \end{align*}
  and transition scheme
  \begin{align*}
    M_{t+1} = s_\tpo M_t + \eps_\tpo,
  \end{align*}
  where the sign $s_t$ is negative at change-points
  \begin{align*}
    s_\tpo = \begin{cases}
      -1 & \quad t+1=T^B_k \text{ for some } k=1,2,\dots,\\
      1 &\quad \text{else},
    \end{cases}
  \end{align*}
  and the independent innovations $\eps_\tpo$ are drawn from either
  $G_+$ or $G_-$ according to the position of $t+1$ within the
  intervals between change-points
  \begin{align*}
    \eps_\tpo \sim \begin{cases}
      G_+ & \quad t+1 < T^B_{1} \text{ or } T^B_k \leq t+1 < T^B_{k+1}, \, k \text{ even},\\
      G_- &\quad \phantom{t+1 < T^B_{1} \text{ or }} T^B_k \leq t+1 <
      T^B_{k+1}, \, k \text{ odd}.
    \end{cases}
  \end{align*}
\end{example}

\begin{remark}
  An alternative tail chain approach to Example~\ref{ex:arch} is to
  square the ARCH process, $Y_t^2$ instead of $Y_t$, which leads to a
  random walk tail chain as discussed in \cite{resnzebe13a}.  An
  advantage of our approach is that we may condition on an upper (or
  by symmetry lower) extreme state whereas in the squared process this
  information is lost and one has to condition on its norm being
  large.
\end{remark}

\subsection{Negative dependence}

In the previous examples the change from upper to lower extremes and
vice versa has been driven by a latent Bernoulli random variable.  If
the consecutive states of a time series are negatively dependent, such
switchings are almost certain.  An example is the autoregressive
Gaussian Markov chain in Example~\ref{ex:marginalscale}, in which case
the tail chain representation there trivially remains true even if the
correlation parameter $\rho$ varies in the negatively dependent regime
$(-1,0)$. More generally, our previous results may be transferred to
Markov chains with negatively dependent consecutive states when
interest lies in both upper extreme states and lower extreme states.
For instance, the conditions for Theorem~\ref{thm:tailchain} may be
adapted as follows.

\begin{description}[wide=0\parindent]
\item[Assumption C1]
  \emph{(behaviour of the next state as the previous state becomes extreme)}\\
  There exist measurable norming functions $\locpm(v), \locmp(v) \in
  \RR$, $\scalepm(v), \scalemp(v)>0$ and non-degenerate distribution
  functions $K_-$, $K_+$ on $\RR$, such that
  \begin{align*}
    &\pi(v, \locpm(v) + \scalepm(v) dx) \wk K_-(dx) \qquad \text{as }
    v
    \uparrow \infty,\\
    &\pi(v, \locmp(v) + \scalemp(v) dx) \wk K_+(dx) \qquad \text{as }
    v \downarrow -\infty.
  \end{align*}
\item[Assumption C2] \emph{(norming functions and update functions for
    the tail chain)}
  \begin{enumerate}[label={(\alph*)}]
  \item Additionally to $\loc_1=\locpm$ and $\scale_1=\scalepm$ assume
    there exist measurable norming functions $\loc_t(v) \in \RR$,
    $\scale_t(v)>0$ for $t=2,3,\dots$, such that, for all $x \in \RR$,
    $t=1,2,\dots$
    \begin{align*}
      \loc_t(v)+\scale_t(v)x \rightarrow \left\{ \begin{array}{ll}
          -\infty & \quad \text{$t$ odd,}\\ \infty & \quad \text{$t$
            even,} \end{array}\right.  \qquad \text{as } v \uparrow
      \infty.
    \end{align*}
  \item Set
    \begin{align*}
      \locOEtime=
      \begin{cases}
        \loc_+ &\quad t \text{ odd},\\
        \loc_- &\quad t \text{ even},
      \end{cases}
      \quad \text{and} \quad \scaleOEtime=
      \begin{cases}
        \scale_+ &\quad t \text{ odd},\\
        \scale_- &\quad t \text{ even}.
      \end{cases}
    \end{align*}
    and assume further that there exist continuous update functions
    \begin{align*}
      \psi_\tpo^\loc(x) &= \lim_{v\rightarrow \infty}\frac{
        \locOEtime\left(\loc_t(v)+\scale_t(v)x\right)
        -\loc_{t+1}(v)}{\scale_{t+1}(v)} \in \RR,
      \\
      \psi_\tpo^\scale(x) &= \lim_{v\rightarrow
        \infty}\frac{\scaleOEtime\left(\loc_t(v)+\scale_t(v)x\right)}{\scale_{t+1}(v)}
      > 0,
    \end{align*}
    defined for $x \in \RR$ and $t=1,2,\dots$, such that the remainder
    terms
    \begin{align*}
      r^\loc_\tpo(v,x)&= \frac{\loc_{t+1}(v) -
        \locOEtime(\loc_t(v)+\scale_t(v)x) +\scale_{t+1}(v)
        \psi^\loc_\tpo(x)}{\scaleOEtime(\loc_t(v)+\scale_t(v)x)},\\
      r^\scale_\tpo(v,x)&= 1-\frac{\scale_{t+1}(v)
        \psi^\scale_\tpo(x)}{\scaleOEtime(\loc_t(v)+\scale_t(v)x)}
    \end{align*}
    converge to $0$ as $v\to \infty$ and both convergences hold
    uniformly on compact sets in the variable $x \in \RR$.
  \end{enumerate}
\end{description}

Using the proof of Theorem~\ref{thm:tailchain}, it is straightforward
to check that the following version adapted to negative dependence
holds true.

\begin{theorem}
  \label{thm:tailchain:negdep}
  Let $\{X_t \,:\, t = 0,1,2, \dots \}$ be a homogeneous Markov chain
  satisfying assumption \fz\,, \textbf{C1} and \textbf{C2}.  Then, as
  $u \uparrow \infty$,
  \begin{align*}
    \left(\frac{X_0 -
        u}{\sigma(u)},\frac{X_1-\loc_1(X_0)}{\scale_1(X_0)},\frac{X_2-\loc_2(X_0)}{\scale_2(X_0)},\dots,\frac{X_t-\loc_t(X_0)}{\scale_t(X_0)}\right)
    \,\bigg\vert\, X_0 > u
  \end{align*}
  converges weakly to $\left(E_0,M_1,M_2,\dots,M_t\right)$, where
  \begin{enumerate}[label={(\roman*)}]
  \item $E_0 \sim H_0$ and $(M_1,M_2,\dots,M_t)$ are independent,
  \item $M_1 \sim K_-$ and $M_{t+1} =
    \psi^\loc_\tpo(M_j)+\psi^\scale_\tpo(M_t) \, \eps_\tpo, \,
    t=1,2,\dots$ for an independent
    sequence 
    of innovations
    \begin{align*}
      \eps_\tpo \sim \left\{ \begin{array}{ll} K_+ & \quad \text{$t$
            odd,} \\ K_- & \quad \text{$t$
            even.} \end{array} \right.
    \end{align*}
  \end{enumerate}
\end{theorem}

\begin{remark} 
  Due to different limiting behaviour of upper and lower tails, the
  tail chain $\{M_t \,:\, t = 0,1,2, \dots \}$ from
  Theorem~\ref{thm:tailchain:negdep} has a second source of potential
  non-homogeneity, since the innovations $\eps_\tpo$ will be generally
  not i.i.d.\ anymore, cf.\ also Remark~\ref{rk:nonhom}.
\end{remark}

\begin{example} \label{ex:ht-neg}
  \emph{(Heffernan-Tawn normalization in case of negative dependence)}\\
  Consider a stationary Markov chain with standard Laplace
  margins~\eqref{eq:lap_mar} and transition kernel $\pi$ satisfying
  \begin{align*}
    &\pi(v, \alpha_- v + |v|^\beta dx) \wk K_-(dx) \qquad
    \text{as } v \uparrow \infty,\\
    &\pi(v, \alpha_+ v + |v|^\beta dx) \wk K_+(dx) \qquad \text{as } v
    \downarrow -\infty.
  \end{align*}
  for some $\alpha_- , \alpha_+ \in (-1,0)$ and $\beta \in [0,1)$.
  Then the normalization after $t$ steps
  \begin{align*}
    a_t(v) &= \left\{ \begin{array}{ll} \alpha_-^{(t+1)/2}
        \alpha_+^{(t-1)/2} v & \quad \text{$t$ odd,} \medskip\\
        \alpha_-^{t/2} \alpha_+^{t/2} v & \quad \text{$t$
          even,} \end{array} \right. \\ b_t(v)&=|v|^\beta
  \end{align*}
  yields the tail chain
  \begin{align*}
    M_{t+1} = \left\{ \begin{array}{ll} \alpha_+ M_t + \left\lvert
          \alpha_-^{(t+1)/2} \alpha_+^{(t-1)/2}\right\rvert^\beta
        \varepsilon_\tpo^+ & \quad \text{$t$ odd,} \medskip\\ \alpha_-
        M_t + \left\lvert\alpha_-^{t/2}
          \alpha_+^{t/2}\right\rvert^\beta \varepsilon_\tpo^- & \quad
        \text{$t$ even,} \end{array} \right.
  \end{align*}
  with independent innovations $\eps_\tpo^+ \sim K_+$ and $\eps_\tpo^-
  \sim K_-$.
\end{example}

\begin{example} \emph{(negatively dependent Gaussian transition kernel with Laplace margins)}\\
  Consider as in Example~\ref{ex:marginalscale} a stationary Gaussian
  Markov chain with standard Laplace margins and $\rho\in (-1,0)$.
  Assumption \textbf{C1} is satisfied with $\locpm(v) = \locmp(v) =
  -\rho^2 v$, $\scalepm(v)= \scalemp(v)=v^{1/2}$ and
  $K(x)=K_-(x)=K_+(x)=
  \Phi\left(x/(2\rho^2(1-\rho^2))^{1/2}\right)$. Then the
  normalization after $t$ steps $a_t(v) = (-1)^t\rho^{2t} v$ and
  $b_t(v) = |v|^\beta$ yields the tail chain $M_{t+1} = -\rho^2 M_t +
  (-\rho)^{t}\varepsilon_\tpo$ with independent innovations $\eps_\tpo
  \sim K$.
\end{example}

\begin{remark}
  If the $\beta$-parameter of the Heffernan-Tawn normalization in
  Example~\ref{ex:ht-neg} is different for lower and upper extreme
  values, one encounters similar varieties of different behaviour as
  in Example~\ref{ex:ht-mix}.
\end{remark}

\section{Proofs}\label{sec:proofs}

\subsection{Proofs for Section~\ref{sec:tailchain}}

Some techniques in the followings proofs are analogous to
\cite{kulisoul} with adaptions to our situation including the random
norming as in \cite{janssege14}.  By contrast to previous accounts, we
have to control additional remainder terms, which make the auxiliary
Lemma~\ref{lemma:composition:locunif:cvg} necessary.  The following
result is a preparatory lemma and the essential part of the induction
step in the proof of Theorem~\ref{thm:tailchain}.

\begin{lemma}\label{lemma:for:induction:step} 
  Let $\{X_t \,:\, t = 0,1,2, \dots \}$ be a homogeneous Markov chain
  satisfying assumptions \textbf{A1} and \textbf{A2}. Let $g \in
  C_b(\RR)$. Then, for $t=1,2,\dots$, as $v \uparrow \infty$,
  \begin{align}\label{eq:for:induction:step}
    \int_\RR g(y) \pi(\loc_t(v) +\scale_{t}(v) x , \loc_{t+1}(v)
    +\scale_{t+1}(v) dy ) \rightarrow \int_\RR
    g(\psi^\loc_\tpo(x)+\psi^\scale_\tpo(x)y) K(dy)
  \end{align}
  and the convergence holds uniformly on compact
  sets in the variable $x \in \RR$.
\end{lemma}

\begin{proof} Let us fix $t \in \NN$. We start by noticing
  \begin{align*}
    &\loc_{t+1}(v)+\scale_{t+1}(v) y  \\
    &= \loc(\loc_t(v)+\scale_t(v)x) + \scale(\scale_t(v)x+\loc_t(v))\,
    \left[ r^\loc_\tpo(v,x) +\left(1-r^\scale_\tpo(v,x)\right) \,
      \frac{y - \psi^\loc_\tpo(x)}{\psi^\scale_\tpo(x)} \right].
  \end{align*}
  Hence the left-hand side of (\ref{eq:for:induction:step}) can be
  rewritten as
  \begin{align*}
    &\int_\RR g(y) \pi(\loc_t(v) +\scale_{t}(v) x, \loc_{t+1}(v)+\scale_{t+1}(v) dy )\\
    &= \int_\RR
    g\left(\psi^\loc_\tpo(x)+\psi^\scale_\tpo(x)\frac{y-r^\loc_\tpo(v,x)}{1-r^\scale_\tpo(v,x)}\right)
    \pi(A_t(v,x), \loc(A_t(v,x)) + \scale(A_t(v,x))\, dy
    ))\\
    &= \int_{\RR} f_v(x,y) \,\pi_{v,x}(dy)
  \end{align*}
  if we abbreviate
  \begin{align*}
    A_t(v,x)&=\loc_t(v)+\scale_t(v)x,  \\
    \pi_x(dy)&=\pi(x,\loc(x)+\scale(x)\,dy),\\
    \pi_{v,x}(dy)&=\pi_{A_t(v,x)}(dy),\\
    f(x,y)&=g\left(\psi^\loc_\tpo(x)+\psi^\scale_\tpo(x)y\right),\\
    f_v(x,y)&=f\left(x,\frac{y-r^\loc_\tpo(v,x)}{1-r^\scale_\tpo(v,x)}\right),
  \end{align*}
  and we need to show that for compact $C \subset \RR$
  \begin{align*}
    \sup_{x \in C} \left\lvert \int_{\RR} f_v(x,y) \pi_{v,x}(dy) -
      \int_{\RR} f(x,y) K(dy) \right\rvert \rightarrow 0 \qquad
    \text{as } v \uparrow \infty.
  \end{align*}
  In particular it suffices to show the slightly more general
  statement that 
  \begin{align*}
    \sup_{c_1 \in C_1} \sup_{c_2 \in C_2} \left\lvert \int_{\RR}
      f_v(c_1,y) \pi_{v,c_2}(dy) - \int_{\RR} f(c_1,y) K(dy)
    \right\rvert \rightarrow 0 \qquad \text{as } v \uparrow \infty,
  \end{align*}
  for compact sets $C_1,C_2 \subset \RR$. Using the inequality
  \begin{align*}
    & \left\lvert \int_{\RR} f_v(c_1,y) \pi_{v,c_2}(dy) - \int_{\RR}
      f(c_1,y) K(dy)
    \right\rvert\\
    & \leq \int_{\RR} \left\lvert f_v(c_1,y) - f(c_1,y) \right\rvert
    \pi_{v,c_2}(dy) + \left\lvert \int_{\RR} f(c_1,y) \pi_{v,c_2}(dy)
      - \int_{\RR} f(c_1,y) K(dy) \right\rvert,
  \end{align*}
  the preceding statement will follow from the following two steps.\\
  \textbf{1st step} We show
  \begin{align*}
    \sup_{c_2 \in C_2} \int_{\RR} \left(\sup_{c_1 \in C_1} \left\lvert
        f_v(c_1,y) - f(c_1,y) \right\rvert\right)\, \pi_{v,c_2}(dy)
    \rightarrow 0 \qquad \text{as } v \uparrow \infty.
  \end{align*}
  Let $\eps>0$ and let $M$ be an upper bound for $g$, such that $2M$
  is an upper bound for $\lvert f_v-f\rvert$. Due to assumption
  \textbf{A1} and Lemma~\ref{lemma:tightness} there exists
  $L=L_{\eps,M} \in \RR$ and a compact set $C=C_{\eps,M} \subset \RR$,
  such that $\pi_\ell(C) > 1-\eps/(2M)$ for all $\ell \geq L$.
  Because of assumption \textbf{A2} (a) there exists $V=V_L \in \RR$
  such that $A_t(v,c_2) \geq A_t(v,\min(C_2)) \geq L$ for all $v \geq
  V$, $c_2 \in C_2$.  Hence
  \begin{align*}
    \pi_{v,c_2}(C)>1-\eps/(2M) \quad \text{ for all } v \geq V, \, c_2
    \in C_2.
  \end{align*}
  Moreover, by assumption \textbf{A2} (b) the map
  \begin{align*}
    \RR \times \RR \ni (x,y) \mapsto
    \psi^\loc_\tpo(x)+\psi^\scale_\tpo(x)\frac{y-r^\loc_\tpo(v,x)}{1-r^\scale_\tpo(v,x)}
    \in \RR
  \end{align*}
  converges uniformly on compact sets to the map
  \begin{align*}
    \RR \times \RR \ni (x,y) \mapsto
    \psi^\loc_\tpo(x)+\psi^\scale_\tpo(x)y \in \RR.
  \end{align*}
  Since the latter map is continuous by assumption \textbf{A2} (b) (in
  particular it maps compact sets to compact sets) and since $g$ is
  continuous, Lemma~\ref{lemma:composition:locunif:cvg} implies that
  \begin{align*}
    \sup_{y \in C} \varphi_v(y) \to 0 \quad \text{ for } \quad
    \varphi_v(y)=\sup_{c_1 \in C_1} \left\lvert f_v(c_1,y) - f(c_1,y)
    \right\rvert \quad \text{ as } \quad v \uparrow \infty.
  \end{align*}
  The hypothesis of the 1st step follows now from
  \begin{align*}
    \sup_{c_2 \in C_2} \int_{\RR} \varphi_v(y) \, \pi_{v,c_2}(dy)
    &\leq \sup_{c_2 \in C_2} \left(\int_{C} \varphi_v(y)\, \pi_{v,c_2}(dy)  + \int_{\RR \setminus C} \varphi_v(y) \,\pi_{v,c_2}(dy) \right)\\
    &\leq \sup_{y \in C} \varphi_{v}(y) \cdot 1 + 2M \cdot
    \varepsilon/(2M).
  \end{align*}
  \textbf{2nd step} We show
  \begin{align*}
    \sup_{c_2 \in C_2} \sup_{c_1 \in C_1} \left\lvert \int_{\RR}
      f(c_1,y) \pi_{v,c_2}(dy) - \int_{\RR} f(c_1,y) K(dy)
    \right\rvert \rightarrow 0 \qquad \text{as } v \uparrow \infty.
  \end{align*}
  Let $\eps >0$. Because of assumption \textbf{A1} and
  Lemma~\ref{lemma:scmapping} (ii) there exists $L=L_\eps\geq 0$, such
  that
  \begin{align*}
    \sup_{c_1 \in C_1} \left\lvert \int_{\RR} f(c_1,y) \pi_{\ell}(dy)
      - \int_{\RR} f(c_1,y) K(dy) \right\rvert <\eps \quad \text{ for
      all } \ell \geq L
  \end{align*}
  Because of assumption \textbf{A2} (a) there exists $V=V_L \in \RR$
  such that $A_t(v,c_2) \geq A_t(v,\min(C_2)) \geq L$ for all $v \geq
  V$, $c_2 \in C_2$. Hence, as desired,
  \begin{align*}
    \sup_{c_1 \in C_1} \left\lvert \int_{\RR} f(c_1,y) \pi_{v,c_2}(dy)
      - \int_{\RR} f(c_1,y) K(dy) \right\rvert <\eps \quad \text{ for
      all } v \geq V, \, c_2 \in C_2.
  \end{align*}
\end{proof}

\paragraph{Proof of Theorem~\ref{thm:tailchain}}

\begin{proof}[Proof of Theorem~\ref{thm:tailchain}]
  To simplify the notation, we abbreviate the affine transformations
  \begin{align*}
    v_u(y_0)=u+\sigma(u)y_0 \qquad \text{and} \qquad
    A_t(v,y)=\loc_t(v)+\scale_t(v)y, \qquad t=1,2,\dots
  \end{align*}
  henceforth. Considering the measures
  \begin{align*}
    &\mu^{(u)}_t(dy_0,\dots,dy_t)\\
    &= \pi(A_{t-1}(v_u(y_0),y_{t-1}),A_t(v_u(y_0),dy_t)) \dots \pi(v_u(y_0),A_{1}(v_u(y_0),dy_1)) \frac{F_0(v_u(dy_0))}{\overline{F}_0(u)},\\
    &\mu_t(dy_0,\dots,dy_t)\\
    &=K\left(\frac{dy_t-\psi^\loc_{t-1}(y_{t-1})}{\psi^\scale_{t-1}(y_{t-1})}\right)
    \dots
    K\left(\frac{dy_2-\psi^\loc_{1}(y_{1})}{\psi^\scale_{1}(y_{1})}\right)
    K(dy_1) H_0(dy_0),
  \end{align*}
  on $[0,\infty) \times \RR^t$, we may rewrite
  \begin{align*}
    &\EE \left[ f\left(\frac{X_0 - u}{\sigma(u)},\frac{X_1-\loc_1(X_0)}{\scale_1(X_0)},\dots,\frac{X_t-\loc_t(X_0)}{\scale_t(X_0)}\right) \,\bigg\vert\, X_0 > u\right]\\
    &= \int_{[0,\infty) \times \RR^t} f\left(y_0,y_1,\dots,y_t\right)
    \mu^{(u)}_t(dy_0,\dots,dy_t)
  \end{align*}
  and
  \begin{align*}
    \EE \left[ f\left(E_0,M_1,\dots,M_t\right) \right] =
    \int_{[0,\infty) \times \RR^t} f\left(y_0,y_1,\dots,y_t\right)
    \mu_t(dy_0,\dots,dy_t)
  \end{align*}
  for $f \in C_b([0,\infty)\times \RR^t)$. We need to show that
  $\mu^{(u)}_t(dy_0,\dots,dy_t)$ converges weakly to
  $\mu_t(dy_0,\dots,dy_t)$. The proof is by induction on $t$. 

  For $t=1$ it suffices to show that for $f_0 \in C_b([0,\infty))$ and
  $g \in C_b(\RR)$
  \begin{align}\label{eq:induction:start:lhs}
    \notag&\int_{[0,\infty)\times \RR} f_0(y_0) g(y_1) \mu^{(u)}_1(dy_0,dy_1)\\
    &=\int_{[0,\infty)} f_0(y_0) \left[\int_\RR g(y_1)
      \pi(v_u(y_0),A_1(v_u(y_0),dy_1)) \right]
    \frac{F_0(v_u(dy_0))}{\overline{F}_0(u)}
  \end{align}
  converges to $\int_{[0,\infty)\times \RR} f_0(y_0) g(y_1)
  \mu_1(dy_0,dy_1)=\EE(f_0(E_0)) \EE(g(M_1))$. The term in the inner
  brackets $[ \dots ]$ is bounded and, by assumption \textbf{A1}, it
  converges to $\EE(g(M_1))$ for $u \uparrow \infty$, since $v_u(y_0)
  \to \infty$ for $u \uparrow \infty$. The convergence holds even
  uniformly in the variable $y_0 \in [0,\infty)$, since $\sigma(u)>0$.
  Therefore, Lemma~\ref{lemma:scmapping} (i) applies, which guarantees
  convergence of the entire term (\ref{eq:induction:start:lhs}) to
  $\EE(f_0(E_0)) \EE(g(M_1))$ with regard to assumption \fz.
   
  Now, let us assume, the statement is proved for some $t \in \NN$.
  It suffices to show that for $f_0 \in C_b([0,\infty)\times \RR^t)$,
  $g\in C_b(\RR)$
  \begin{align}\label{eq:induction:step}
    \notag &\int_{[0,\infty)\times \RR^{t+1}} f_0(y_0,y_1,\dots,y_t) g(y_{t+1}) \mu^{(u)}_{t+1}(dy_0,dy_1,\dots,dy_t,dy_{t+1})\\
    \notag &=\int_{[0,\infty)\times \RR^{t}} f_0(y_0,y_1,\dots,y_t) \left[\int_\RR g(y_{t+1}) \pi(A_{t}(v_u(y_0),y_t),A_{t+1}(v_u(y_0),dy_{t+1})) \right]\\
    &\hspace{9.5cm}\mu^{(u)}_{t}(dy_0,dy_1,\dots,dy_t)
  \end{align}
  converges to
  \begin{align}\label{eq:induction:step:limit}
    \notag&\int_{[0,\infty)\times \RR^{t+1}} f_0(y_0,y_1,\dots,y_t) g(y_{t+1}) \mu_{t+1}(dy_0,dy_1,\dots,dy_t,dy_{t+1})\\
    \notag &=\int_{[0,\infty)\times \RR^{t}} f_0(y_0,y_1,\dots,y_t) \left[\int_\RR g(y_{t+1}) K\left(\frac{dy_{t+1}-\psi_\tpo^\loc(y_t)}{\psi_\tpo^\scale(y_t)}\right) \right]\\
    &\hspace{9.5cm} \mu_{t}(dy_0,dy_1,\dots,dy_t).
  \end{align}
  The term in square brackets of (\ref{eq:induction:step}) is bounded
  and, by Lemma~\ref{lemma:for:induction:step} and assumptions
  \textbf{A1} and \textbf{A2}, it converges uniformly on compact sets
  in the variable $y_{t}$ to the continuous function $\int_{\RR}
  g(\psi^\loc_\tpo(y_t)+\psi_\tpo^\scale(y_t) y_{t+1}) K(dy_{t+1})$
  (the term in square brackets of
  (\ref{eq:induction:step:limit})). This convergence holds uniformly
  on compact sets in both variables $(y_0,y_t) \in [0,\infty)\times
  \RR$ jointly, since $\sigma(u)>0$. Hence, the induction hypothesis
  and Lemma~\ref{lemma:scmapping} (i) imply the desired result.
\end{proof}

\begin{remark}\label{rk:tailchain:relaxed:proof}
  Under the relaxed assumptions of Remark~\ref{rk:tailchain:relaxed},
  the proof of Theorem~\ref{thm:tailchain} can be modified by
  replacing the integration area $\RR^t$ by $\overline{S}_1 \times
  \dots \times \overline{S}_t$ and by letting $x$ vary in
  $\overline{S}_t$ and $y \in \overline{S}_{t+1}$ in
  Lemma~\ref{lemma:for:induction:step}.
\end{remark}

The following lemma is a straightforward analogue to
Lemma~\ref{lemma:for:induction:step} and prepares the induction step
for the proof of Theorem~\ref{thm:tailchain:nonneg}. We omit its
proof, since the only changes compared to the proof of
Lemma~\ref{lemma:for:induction:step} are the removal of the location
normings and the fact that $x$ varies in $[\delta,\infty)$ instead of
$\RR$ and $y$ in $[0,\infty)$ instead of $\RR$.

\begin{lemma}
  \label{lemma:for:induction:step:nonneg}
  Let $\{X_t \,:\, t = 0,1,2, \dots \}$ be a non-negative homogeneous
  Markov chain satisfying assumptions \textbf{B1} and \textbf{B2} (a)
  and (b).  Let $g \in C_b([0,\infty))$. Then, as $v \uparrow \infty$,
  \begin{align}\label{eq:for:induction:step:nonneg}
    \int_{[0,\infty)} g(y) \pi(\scale_{t}(v) x,\scale_{t+1}(v) dy )
    \rightarrow \int_{[0,\infty)} g(\psi^\scale_\tpo(x)y) K(dy)
  \end{align}
  for $t=1,2,\dots$ and the convergence holds uniformly on compact
  sets in the variable $x \in [\delta,\infty)$ for any $\delta >0$.
\end{lemma}

\paragraph{Proof of Theorem~\ref{thm:tailchain:nonneg}}

Even though parts of the following proof resemble the proof of
Theorem~\ref{thm:tailchain}, one has to control the mass at 0 of the
limiting measures in this setting. Therefore, a second induction
hypothesis (II) enters the proof.

\begin{proof}[Proof of Theorem~\ref{thm:tailchain:nonneg}]
  To simplify the notation, we abbreviate the affine transformation
  $v_u(y_0)=u+\sigma(u)y_0$
  henceforth. Considering the measures
  \begin{align}
    \notag &\mu^{(u)}_t(dy_0,\dots,dy_t)\\
    \label{eq:firstmeasure:nonneg} &= \pi(b_{t-1}(v_u(y_0))y_{t-1},b_t(v_u(y_0))dy_t) \dots \pi(v_u(y_0),b_{1}(v_u(y_0))dy_1) \frac{F_0(v_u(dy_0))}{\overline{F}_0(u)},\\
    \notag &\mu_t(dy_0,\dots,dy_t)\\
    \label{eq:secondmeasure:nonneg} &=
    K\left(\frac{dy_t}{\psi^\scale_{t-1}(y_{t-1})}\right) \dots
    K\left(\frac{dy_2}{\psi^\scale_{1}(y_{1})}\right) K(dy_1)
    H_0(dy_0),
  \end{align}
  on $[0,\infty) \times [0,\infty)^t$, we may rewrite
  \begin{align*}
    &\EE \left[ f\left(\frac{X_0 - u}{\sigma(u)},\frac{X_1}{\scale_1(X_0)},\dots,\frac{X_t}{\scale_t(X_0)}\right) \,\bigg\vert\, X_0 > u\right]\\
    &= \int_{[0,\infty) \times [0,\infty)^t}
    f\left(y_0,y_1,\dots,y_t\right) \mu^{(u)}_t(dy_0,\dots,dy_t)
  \end{align*}
  and
  \begin{align*}
    \EE \left[ f\left(E_0,M_1,\dots,M_t\right) \right] =
    \int_{[0,\infty) \times [0,\infty)^t}
    f\left(y_0,y_1,\dots,y_t\right) \mu_t(dy_0,\dots,dy_t)
  \end{align*}
  for $f \in C_b([0,\infty)\times [0,\infty)^t)$.  In particular note
  that $b_j(0)$, $j=1,\dots,t$ need not be defined in
  (\ref{eq:firstmeasure:nonneg}), since $v_u(y_0)\geq u>0$ for
  $y_0\geq 0$ and sufficiently large $u$, whereas
  (\ref{eq:secondmeasure:nonneg}) is well-defined, since $K$ puts no
  mass to $0 \in [0,\infty)$.  Formally, we may set
  $\psi^\scale_{j}(0)=1$, $j=1,\dots,t$ in order to emphasize that we
  consider measures on $[0,\infty)^{t+1}$ here (instead of
  $[0,\infty)\times (0,\infty)^t$).  To prove the theorem, we need to
  show that $\mu^{(u)}_t(dy_0,\dots,dy_t)$ converges weakly to
  $\mu_t(dy_0,\dots,dy_t)$.  The proof is by induction on $t$.  In
  fact, we show two statements ((I) and (II)) by induction on $t$:
  \begin{enumerate}[label={(\Roman*)}]
  \item $\mu^{(u)}_t(dy_0,\dots,dy_t)$ converges weakly to
    $\mu_t(dy_0,\dots,dy_t)$ as $u \uparrow \infty$.
  \item For all $\varepsilon>0$ there exists $\delta_t>0$ such that
    $\mu_t([0,\infty)\times [0,\infty)^{t-1} \times
    [0,\delta_t])<\eps$.
  \end{enumerate}

  \noindent \underline{(I) for $t=1$:} It suffices to show that for
  $f_0 \in C_b([0,\infty))$ and $g \in C_b([0,\infty))$
  \begin{align}\label{eq:induction:start:lhs:nonneg}
    \notag&\int_{[0,\infty)\times [0,\infty)} f_0(y_0) g(y_1) \mu^{(u)}_1(dy_0,dy_1)\\
    &=\int_{[0,\infty)} f_0(y_0) \left[\int_{[0,\infty)} g(y_1)
      \pi(v_u(y_0),b_1(v_u(y_0))dy_1) \right]
    \frac{F_0(v_u(dy_0))}{\overline{F}_0(u)}
  \end{align}
  converges to $\int_{[0,\infty)\times [0,\infty)} f_0(y_0) g(y_1)
  \mu_1(dy_0,dy_1)=\EE(f_0(E_0)) \EE(g(M_1)).$ The term in the inner
  brackets $[\dots]$ is bounded and, by assumption \textbf{B1}, it
  converges to $\EE(g(M_1))$ for $u \uparrow \infty$, since $v_u(y_0)
  \to \infty$ for $u \uparrow \infty$. The convergence holds even
  uniformly in the variable $y_0 \in [0,\infty)$, since $\sigma(u)>0$.
  Therefore, Lemma~\ref{lemma:scmapping} (i) applies, which guarantees
  convergence of the entire term (\ref{eq:induction:start:lhs:nonneg})
  to $\EE(f_0(E_0)) \EE(g(M_1))$ with regard to assumption
  $\mathbf{F_0}$.

  \noindent \underline{(II) for $t=1$:}
  Note that $K(\{0\})=0$. Hence, there exists $\delta>0$ such that $K([0,\delta])<\eps$, which immediately entails $\mu_1([0,\infty)\times[0,\delta])=H_0([0,\infty))K([0,\delta])<\delta$.

  \noindent Now, let us assume that both statements ((I) and (II)) are
  proved for some $t \in \NN$.  

  \noindent \underline{(I) for $t+1$:} It suffices to show that for
  $f_0 \in C_b([0,\infty)\times [0,\infty)^t)$, $g\in C_b([0,\infty))$
  \begin{align}\label{eq:induction:step:nonneg}
    \notag &\int_{[0,\infty)\times [0,\infty)^{t+1}} f_0(y_0,y_1,\dots,y_t) g(y_{t+1}) \mu^{(u)}_{t+1}(dy_0,dy_1,\dots,dy_t,dy_{t+1})\\
    \notag &=\int_{[0,\infty)\times [0,\infty)^{t}} f_0(y_0,y_1,\dots,y_t) \left[\int_{[0,\infty)} g(y_{t+1}) \pi(b_{t}(v_u(y_0))y_t,b_{t+1}(v_u(y_0))dy_{t+1}) \right]\\
    &\hspace{9.5cm}\mu^{(u)}_{t}(dy_0,dy_1,\dots,dy_t)
  \end{align}
  converges to
  \begin{align}\label{eq:induction:step:limit:nonneg}
    \notag&\int_{[0,\infty)\times [0,\infty)^{t+1}} f_0(y_0,y_1,\dots,y_t) g(y_{t+1}) \mu_{t+1}(dy_0,dy_1,\dots,dy_t,dy_{t+1})\\
    \notag &=\int_{[0,\infty)\times [0,\infty)^{t}} f_0(y_0,y_1,\dots,y_t) \left[\int_{[0,\infty)}g(y_{t+1}) K\left(dy_{t+1}/\psi_\tpo^\scale(y_t)\right) \right]\\
    &\hspace{9.5cm} \mu_{t}(dy_0,dy_1,\dots,dy_t).
  \end{align}

  From Lemma~\ref{lemma:for:induction:step:nonneg} and assumptions
  \textbf{B1} and \textbf{B2} (a) and (b) we know that, for any
  $\delta>0$, the (bounded) term in the brackets $[\dots]$ of
  (\ref{eq:induction:step:nonneg}) converges uniformly on compact sets
  in the variable $y_{t} \in [\delta,\infty)$ to the continuous
  function $\int_{[0,\infty)} g(\psi_\tpo^\scale(y_t) y_{t+1})
  K(dy_{t+1})$ (the term in the brackets $[\dots]$ of
  (\ref{eq:induction:step:limit:nonneg})). This convergence holds even
  uniformly on compact sets in both variables $(y_0,y_t) \in
  [0,\infty)\times [\delta,\infty)$ jointly, since $\sigma(u)>0$.
  Hence, the induction hypothesis (I) and Lemma~\ref{lemma:scmapping}
  (i) imply that for any $\delta>0$ the integral in
  (\ref{eq:induction:step:nonneg}) converges to the integral in
  (\ref{eq:induction:step:limit:nonneg}) if the integrals with respect
  to $\mu_t$ and $\mu_t^{(u)}$ were restricted to
  $A_\delta:=[0,\infty)\times [0,\infty)^{t-1} \times [\delta,\infty)$
  (instead of integration over $[0,\infty)\times [0,\infty)^{t-1}
  \times [0,\infty)$).
                
  Therefore (and since $f_0$ and $g$ are bounded) it suffices to
  control the mass of $\mu_t$ and $\mu_t^{(u)}$ on the complement
  $A_\delta^c=[0,\infty)\times [0,\infty)^{t-1} \times [0,\delta)$.
  We show that for some prescribed $\eps>0$ it is possible to find
  some sufficiently small $\delta>0$ and sufficiently large $u$, such
  that $\mu_t(A_\delta^c)<\eps$ and $\mu^{(u)}_t(A^c_\delta)<2\eps$.
  Because of the induction hypothesis (II), we have indeed
  $\mu_t(A_{\delta_t})<\eps$ for some $\delta_t>0$.  Choose
  $\delta=\delta_t/2$ and note that the sets of the form $A_\delta$
  are nested. Let $C_\delta$ be a continuity set of $\mu_t$ with
  $A^c_\delta \subset C_\delta \subset A^c_{2\delta}$.
  Then the value of $\mu_t$ on all three sets $A^c_\delta,C_\delta,A^c_{2\delta}$ is smaller than $\eps$ and because of the induction hypothesis (I), the value $\mu^{(u)}_t(C_\delta)$ converges to $\mu_t(C_\delta)<\eps$. Hence, for sufficiently large $u$, we also have $\mu^{(u)}_t(A^c_\delta)<\mu^{(u)}_{t}(C_\delta)<\mu_t(C_\delta)+\eps<2\eps$, as desired.

  \noindent \underline{(II) for $t+1$:} We have for any $\delta>0$ and
  any $c>0$
  \begin{align*}
    &\mu_{t+1}([0,\infty)\times[0,\infty)^t \times [0,\delta]) =
    \int_{[0,\infty)\times[0,\infty)^t}
    K\left(\left[0,\delta/\psi_\tpo^\scale(y_t)\right]\right)
    \mu_t(dy_0,\dots,dy_t).
  \end{align*}
  Splitting the integral according to $\{\psi_\tpo^\scale(y_t)>c\}$ or
  $\{\psi_\tpo^\scale(y_t)\leq c\}$ yields
  \begin{align*}
    &\mu_{t+1}([0,\infty)\times[0,\infty)^t \times [0,\delta]) \leq
    K\left(\left[0,\delta/c\right]\right) + \mu_t( [0,\infty)\times
    [0,\infty)^{t-1}\times (\psi_\tpo^\scale)^{-1}([0,c])\}).
  \end{align*}
  By assumption \textbf{B2} (c) and the induction hypothesis (II) we
  may choose $c>0$ sufficiently small, such that the second summand
  $\mu_t([0,\infty)\times [0,\infty)^{t-1}\times
  (\psi_\tpo^\scale)^{-1}([0,c])\})$ is smaller than
  $\eps/2$. Secondly, since $K(\{0\})=0$, it is possible to choose
  $\delta_{t+1}=\delta>0$, such that the first
  summand $K\left(\left[0,\frac{\delta}{c}\right]\right)$ is smaller
  than $\eps/2$, which shows (II) for $t+1$.
\end{proof}

\subsection{Auxiliary arguments}

The following lemma is a slight modification of Lemma~6.1.\ of
\cite{kulisoul}.  In the first part (i), we only assume the functions
$\varphi_n$ are measurable (and not necessarily continuous), whereas
we require the limiting function $\varphi$ to be continuous. Since its
proof is almost verbatim the same as in \cite{kulisoul}, we refrain
from representing it here.  The second part (ii) is a direct
consequence of Lemma~6.1.\ of \cite{kulisoul}, cf.\ also
\cite{bill99}, p.~17, Problem~8.

\begin{lemma}
  \label{lemma:scmapping}
  Let $(E,d)$ be a complete locally compact separable metric space and
  $\mu_n$ be a sequence of probability measures which converges weakly
  to a probability measure $\mu$ on $E$.
  \begin{enumerate}[label={(\roman*)}]
  \item Let $\varphi_n$ be a uniformly bounded sequence of measurable
    functions which converges uniformly on compact sets of $E$ to a
    continuous function $\varphi$.  Then $\varphi$ is bounded on $E$
    and $\lim_{n \to \infty} \mu_n(\varphi_n) \rightarrow
    \mu(\varphi)$.
  \item Let $F$ be a topological space. If $\varphi \in C_b(F \times
    E)$, then the sequence of functions $F \ni x \mapsto \int_E
    \varphi(x,y) \mu_n(dy) \in \RR$ converges uniformly on compact
    sets of $F$ to the (necessarily continuous) function $F \ni x
    \mapsto \int_E \varphi(x,y) \mu(dy) \in \RR$.
  \end{enumerate}
\end{lemma}

\begin{lemma}\label{lemma:tightness}
  Let $(E,d)$ be a complete locally compact separable metric space.
  Let $\mu$ be a probability measure and $(\mu_x)_{x \in \RR}$ a
  family of probability measures on $E$, such that every subsequence
  $\mu_{x_n}$ with $x_n\to \infty$ converges weakly to $\mu$. Then,
  for any $\eps>0$, there exists $L \in \RR$ and a compact set $C
  \subset E$, such that $\mu_\ell(C)>1-\eps$ for all $\ell \geq L$.
\end{lemma}

\begin{proof} First note that the topological assumptions on $E$ imply
  that there exists a sequence of nested compact sets $K_1 \subset K_2
  \subset K_3 \subset \dots$, such that $\bigcup_{n \in \NN} K_n=E$
  and each compact subset $K$ of $E$ is contained in some $K_n$.
  
  Now assume that there exists $\delta>0$ such that for all $L\in \RR$
  and for all compact $C \subset E$ there exists an $\ell \geq L$ such
  that $\mu_{\ell}(C)\leq 1-\delta$. It follows that for all $n \in
  \NN$, there exists an $x_n \geq n$, such that $\mu_{x_n}(K_n)\leq
  1-\delta$. Apparently $x_n \uparrow \infty$ as $n \to \infty$. Hence
  $\mu_{x_n}$ converges weakly to $\mu$ and the set of measures
  $\{\mu_{x_n}\}_{n \in \NN}$ is tight, since $E$ was supposed to be
  complete separable metric. Therefore, there exists a compact set $C$
  such that $\mu_{x_n}(C) > 1-\delta$ for all $n \in \NN$. Since $C$
  is necessarily contained in some $K_{n^*}$ for some $n^* \in \NN$,
  the latter contradicts $\mu_{x_{n^*}}(K_{n^*})\leq 1-\delta$.
\end{proof}

\begin{lemma}\label{lemma:composition:locunif:cvg}
  Let $(E,\tau)$ be a topological space and $(F,d)$ a locally compact
  metric space. Let $\varphi_n:E \rightarrow F$ be a sequence of maps
  which converges uniformly on compact sets to a map $\varphi:E
  \rightarrow F$, which satisfies the property that $\varphi(C)$ is
  relatively compact for any compact $C \subset E$.  Then, for any
  continuous $g:F \rightarrow \RR$, the sequence of maps $g \circ
  \varphi_n$ will converge uniformly on compact sets to $g \circ
  \varphi$.
\end{lemma}
\begin{proof}
  Let $\eps>0$ and $C \subset E$ compact. Since $\varphi(C)$ is
  relatively compact, there exists an $r=r_C>0$ such that
  $V_r(\varphi(C))=\{x \in F \,:\, \exists \,c \in C \,:\,
  d(\varphi(c),x)<r\}$ is relatively compact
  \citep[(3.18.2)]{ddn60}. Since $g$ is continuous, its restriction to
  $\overline{V_r(\varphi(C))}$ (the closure of $V_r(\varphi(C))$) is
  uniformly continuous. Hence, there exists $\delta =
  \delta_{\eps,C,r} >0$, such that all points $x,y \in V_r(\varphi(C))
  \subset \overline{V_r(\varphi(C))}$ with $d(x,y)<\delta$ satisfy
  $\lvert g(x) - g(y)\rvert \leq \eps$. Without loss of generality, we
  may assume $\delta < r$.

  By the uniform convergence of the maps $\varphi_n$ to $\varphi$ when
  restricted to $C$ there exists $N=N_{C,\delta} \in \NN$ such that
  $\sup_{c \in C} d(\varphi_n(c),\varphi(c))< \delta < r$ for all $n
  \geq N$, which subsequently implies $\sup_{c \in C} \lvert g \circ
  \varphi_n(c) - g\circ \varphi(c)\rvert \leq \eps$ as desired.
\end{proof}

\subsection{Comment on Section~\ref{sec:extensions}}

In order to show the stated convergences from
Section~\ref{sec:extensions} one can proceed in a similar manner as
for Section~\ref{sec:tailchain}, but with considerable additional
notational effort. A key observation is the modified form of
Lemma~\ref{lemma:scmapping}~(i) (compared to Lemma~6.1~(i) in
\cite{kulisoul}), which allows to involve indicator functions
converging uniformly on compact sets to the constant function 1. For
instance, for Example~\ref{ex:ht-mix}, it is relevant that for a
continuous and bounded function $f$, the expression
$f(x)\mathbf{1}_{(\alpha_1-c)v+v^{\beta}x>0}$ converges uniformly on
compact sets to the continuous function $f(x)$ as $v\uparrow \infty$,
which implies that $\int f(x)\mathbf{1}_{(\alpha_1-c)v+v^{\beta}x>0}
\pi_1(v,\alpha_1 v + v^{\beta_1} dx)$ converges to $\int f(x)
G_1(dx)$. Likewise, the ``\textbf{1st~step}'' in the proof of
Lemma~\ref{lemma:for:induction:step} can be adapted by replacing $f_v$
by its multiplication with an indicator variable converging uniformly
on compact sets to 1.

{\small
\paragraph{Acknowledgements.}
  IP acknowledges funding from the SuSTaIn program - Engineering and
  Physical Sciences Research Council grant EP/D063485/1 - at the
  School of Mathematics of the University of Bristol and AB from the
  Rural and Environment Science and Analytical Services (RESAS)
  Division of the Scottish Government.  KS would like to thank Anja
  Jan{\ss}en for a fruitful discussion on the topic during the EVA
  2015 conference in Ann Arbor.
}

{

}


\end{document}